\documentclass[11pt,reqno]{amsart}

\usepackage[utf8]{inputenc}
\usepackage[english]{babel}
\usepackage{sidecap}                %
\usepackage{booktabs}               %
\usepackage{float}                  %
\usepackage{indentfirst}            %
\usepackage{titlesec}               %
\usepackage[colorlinks=true,linkcolor=blue,citecolor=blue]{hyperref}
\usepackage{xcolor}
\usepackage{csquotes}
\usepackage{fancyhdr}
\usepackage[new]{old-arrows}        %
\usepackage{adjustbox}
\usepackage{makecell}               %
\usepackage{array}                  %
\usepackage{enumitem}
\usepackage{abstract}

\usepackage{amssymb}
\usepackage{amsthm}
\usepackage{dsfont}                 %
\usepackage{mathtools}
\usepackage{mathrsfs}               %
\usepackage{tikz-cd}

\usepackage[backend=bibtex,style=numeric,maxbibnames=99]{biblatex}
\renewbibmacro{in:}{}
\AtBeginBibliography{\sloppy}
\AtEveryBibitem{\clearfield{issue}}
\AtEveryBibitem{\clearfield{number}}
\DeclareFieldFormat[article,periodical]{volume}{\mkbibbold{#1}}
\DefineBibliographyExtras{english}{}
\bibliography{bibtex}

\titleformat{\section}
  {\center\normalfont\scshape}{\thesection.}{0.5em}{}

\titleformat{\subsection}[runin]
  {\normalfont\bfseries\filright}{\thesubsection}{0.5 em}{}

\usepackage[margin=1.1in,headheight=20pt,headsep=14pt,footskip=25pt]{geometry}

\addto\captionsenglish{}

\makeatletter
\def\blfootnote{\xdef\@thefnmark{}\@footnotetext}
\makeatother

\newcommand{\Z}{\mathbb{Z}}
\newcommand{\R}{\mathbb{R}}

\newcommand{\C}{\mathbb{C}}

\newcommand{\g}{\mathfrak{g}}

\newcommand{\m}{\mathcal{M}}

\newcommand{\dbar}{\bar\partial}

\newcommand{\Opoly}[1][p]{\Omega^{#1}_{\mathrm{pol}}}
\newcommand{\Apoly}[1][p,q]{A^{#1}_{\mathrm{pg}}} 
\newcommand{\GC}{G_{\C}}
\newcommand{\Gzo}{G^{0,1}}
\newcommand{\Hrat}{H_{\mathrm{rat}}}
\newcommand{\ind}{\mathrm{ind}}
\newcommand{\Ad}{\mathrm{Ad}}
\newcommand{\pr}{\mathrm{pr}}

\DeclareMathOperator{\id}{id}

\DeclareMathOperator{\Aut}{Aut}

\DeclareMathOperator{\GL}{GL}

\numberwithin{equation}{section}

\theoremstyle{plain}

\newtheorem{theorem}{Theorem}[section]
\newtheorem*{theorem*}{Theorem}

\newtheorem{proposition}[theorem]{Proposition}
\newtheorem{corollary}[theorem]{Corollary}
\newtheorem{lemma}[theorem]{Lemma}
\newtheorem{remark}[theorem]{Remark}
\newtheorem*{conjecture*}{Conjecture}

\theoremstyle{definition}
\newtheorem{definition}[theorem]{Definition}
\newtheorem{construction}[theorem]{Construction}

\theoremstyle{plain}

\newtheorem{fact}[theorem]{Fact}

\makeatletter
\def\author@andify{%
  \nxandlist {\unskip ,\penalty-1 \space\ignorespaces}%
    {\unskip {} \@@and~}%
    {\unskip \penalty-2 \space \@@and~}%
}
\makeatother

\begin{document}

\title[Dolbeault cohomology of nilmanifolds]{Invariant forms compute the Dolbeault cohomology\\ of complex nilmanifolds}

\author[K.~Hasegawa]{Keizo Hasegawa}
\address{Keizo Hasegawa: Department of Mathematics, Graduate School of Science, Osaka University, Toyonaka, Osaka 560-0043, Japan; and Department of Mathematics, Faculty of Education, Niigata University, Ikarashi-nino-cho, Nishi-ku, Niigata 950-2181, Japan}
\email{hasegawa@math.sci.osaka-u.ac.jp}
\email{hasegawa@ed.niigata-u.ac.jp}

\author[L.~Sillari]{Lorenzo Sillari}
\address{Lorenzo Sillari: Dipartimento di Scienze Matematiche, Fisiche e Informatiche, Unit\`a di Matematica e Informatica, Universit\`a degli Studi di Parma, Parco Area delle Scienze 53/A, 43124 Parma, Italy}
\email{lorenzo.sillari@unipr.it}

\author[A.~Tomassini]{Adriano Tomassini}
\address{Adriano Tomassini: Dipartimento di Scienze Matematiche, Fisiche e Informatiche, Unit\`a di Matematica e Informatica, Universit\`a degli Studi di Parma, Parco Area delle Scienze 53/A, 43124 Parma, Italy}
\email{adriano.tomassini@unipr.it}

\maketitle

\begin{abstract}
\textsc{Abstract.} We prove that the inclusion of left-invariant forms into the Dolbeault complex of a compact nilmanifold $M$ endowed with a left-invariant complex structure $J$ induces an isomorphism in cohomology in every bidegree,  settling a long-standing conjecture originating from a question of Console--Fino and Cordero-Fern\'andez-Gray-Ugarte. As consequences, we show that small deformations of $J$ are still invariant, as conjectured by Hasegawa, and that Bott--Chern, Aeppli, and Fr\"olicher invariants are computed by invariant forms and are independent of the lattice, settling a conjecture of Angella on Bott--Chern cohomology.
\end{abstract}

\blfootnote{  \hspace{-0.55cm} 
{\scriptsize 2020 \textit{Mathematics Subject Classification}. Primary: 22E25, 32G05. Secondary: 17B56, 20J06
K.H.\ is supported by JSPS (Japan Society for the Promotion of Science) KAKENHI Grant Number 25K07000. A.T.\ is partially supported by GNSAGA of INdAM. L.S.\ and A.T.\ are partially supported by the Project PRIN 2022 ``Real and Complex Manifolds: Geometry and Holomorphic Dynamics'' (code 2022AP8HZ9). This research was partially granted by University of Parma through the action \textit{Bando di Ateneo 2026 per la ricerca}, project ``Cohomological and metric aspects in non-(hyper)K\"ahler geometry''.
}
}

\setcounter{tocdepth}{1}

\section{Introduction}\label{sec:intro}

A compact nilmanifold endowed with a left-invariant complex structure is called a \emph{complex nilmanifold}. Complex nilmanifolds, which include primary Kodaira surfaces and the Iwasawa manifold, constitute a fundamental source of compact non-K\"ahler manifolds: apart from complex tori, none of them admits a K\"ahler metric \cite{Has89}. Their invariant differential forms provide finite-dimensional models for studying cohomology and deformation theory, and underlie many computations of Dolbeault, Bott--Chern, Aeppli and Fr\"olicher invariants \cite{Ang14,COUV16,Uga07}. A central question is whether these models compute the Dolbeault cohomology of the manifold. An affirmative answer reduces the computation of Dolbeault cohomology, a global analytic problem, to linear algebra that can be carried out by symbolic computation.

Let $M=\Gamma\backslash G$ be a compact nilmanifold, where $G$ is a simply connected nilpotent Lie group with a uniform lattice $\Gamma$. Then, according to Nomizu's theorem \cite{Nom54}, the inclusion of left-invariant differential forms into the de Rham complex of $M$ induces an isomorphism
\[
H^\bullet(\g;\R)\xrightarrow{\;\sim\;}
H^\bullet_{\mathrm{dR}}(M;\R),
\]
where $\g$ is the Lie algebra of $G$, that is, the de Rham cohomology of $M$ is computed by the finite-dimensional Chevalley--Eilenberg complex of $\g$.

The corresponding assertion for Dolbeault cohomology is much more subtle. A left-invariant complex structure $J$ on $G$ endows invariant forms with a bigrading and a differential $\dbar$, and the inclusion of invariant forms induces a natural map
\begin{equation}\label{eq:iota}
\iota\colon H^{p,q}_{\dbar}(\g,J)
\longrightarrow H^{p,q}_{\dbar}(M).
\end{equation}
The question, first raised in the works of Console--Fino \cite{CF01} and Cordero--Fern\'andez--Gray--Ugarte \cite{CFGU00}, is whether $\iota$ is an isomorphism in every bidegree, so that the global Dolbeault cohomology of $M$ can be computed entirely from the finite-dimensional complex determined by $(\g,J)$. Console--Fino proved that $\iota$ is always injective \cite{CF01}, so it suffices to prove surjectivity. This problem has remained as a long-standing conjecture in the field of Complex Geometry for over two decades \cite{RolSpring}.

The conjecture has been established in several important cases: for complex parallelizable nilmanifolds \cite{Sak76}, for nilpotent complex structures \cite{CFGU00, RTW20}, for abelian complex structures and complex structures rational with respect to the lattice \cite{CF01}, and in complex dimension at most three \cite{FRR19,RTW20}. For background on the conjecture and the earlier results and methods, see \cite{RolSpring}.

Many earlier proofs use holomorphic fibrations compatible with the complex structure and the lattice \cite{RolJLMS,RolSpring}, although such fibrations need not exist in general. Fino--Rollenske--Ruppenthal extended this geometric approach to suitable foliations, thereby settling the conjecture in complex dimension three \cite{FRR19}.

Console--Fino also proved that the set of left-invariant complex structures for which the conjecture holds is open \cite{CF01}. Openness alone, however, does not control special limiting structures, where additional Dolbeault classes could appear \cite[Section~3.3]{RolSpring}. A general proof must therefore account both for these limiting structures and for the absence of decompositions compatible with the complex structure and the lattice.

In this paper, we prove this conjecture in full generality, establishing the holomorphic counterpart of Nomizu's theorem.

\begin{theorem}\label{thm:main}
For every compact nilmanifold $M=\Gamma\backslash G$ endowed with a left-invariant complex structure $J$, the inclusion of left-invariant forms induces an isomorphism
\[
\iota\colon H^{p,q}_{\dbar}(\g,J)\xrightarrow{\;\sim\;}H^{p,q}_{\dbar}(M)
\]
for every bidegree $(p,q)$.
\end{theorem}

Every Dolbeault cohomology class therefore admits a left-invariant representative, and the compatibility of the inclusion with wedge products identifies the entire Dolbeault cohomology algebra with that of the invariant complex. In particular, this algebra and all Hodge numbers depend only on $(\g,J)$, not on the lattice $\Gamma$.

The starting point of the proof is a recent theorem of
Hasegawa--Rollenske--Sillari--Tomassini--Wehler \cite{HRSTW26}:
the universal cover $(G,J)$ of a complex nilmanifold $M$ is
biholomorphic to $\C^n$, where $n=\dim_{\C}M$, by a map $\Phi$
that is polynomial in exponential coordinates. We show that its
inverse is also polynomial in real exponential coordinates
(Lemma~\ref{lemma:phiinv}). Transporting the deck action through
$\Phi$ presents $M$ as a quotient
\[
M\cong\rho(\Gamma)\backslash\C^n,\qquad
\rho(\gamma)=\Phi\circ L_\gamma\circ\Phi^{-1}
\in\operatorname{Aut}_{\mathrm{pol}}(\C^n),
\]
of $\C^n$ by polynomial automorphisms of uniformly bounded
degree. On this model, a Cartan--Leray double complex identifies
$H^{p,q}_{\dbar}(M)$ with the group cohomology
$H^q(\Gamma,\Omega^p(\C^n))$ of holomorphic $p$-forms.
The argument splits into two parts, connected to each other
by the group cohomology of the submodule
$\Omega^p_{\mathrm{pol}}\subset\Omega^p(\C^n)$ of forms
with polynomial coefficients.

The essential analytic step is to prove that this inclusion
induces a surjection (Theorem~\ref{thm:E1})
\[
H^q\bigl(\Gamma,\Omega^p_{\mathrm{pol}}(\C^n)\bigr)
\longrightarrow H^q\bigl(\Gamma,\Omega^p(\C^n)\bigr).
\]

A key observation is that $\Gamma$-invariant smooth forms on $\C^n$, corresponding to smooth forms lifted from $M$, have polynomial growth, together with all their derivatives (Lemma~\ref{lemma:frame}). H\"ormander's weighted $L^2$-estimates \cite{Hor90}, combined with interior elliptic estimates, provide $\dbar$-primitives with the same growth property (Lemma~\ref{lemma:pg}). This allows us to carry out a staircase argument in the Cartan--Leray double complex within forms with polynomial growth, producing a group $q$-cocycle with values in holomorphic $(p,0)$-forms of polynomial growth. This cocycle represents the original Dolbeault class under the preceding identification. Cauchy's estimates then show that its values lie in $\Omega^p_{\mathrm{pol}}(\C^n)$ (Lemma \ref{lemma:liouville}).

The remaining step is algebraic. We realize $\C^n$ as the affine homogeneous space $G_{\C}/G^{0,1}$, where $G_{\C}$ is the complexification of $G$ and $G^{0,1}=\exp(\g^{0,1})$. In this realization, polynomial holomorphic $p$-forms constitute an algebraically induced module for the unipotent group $G_{\C}$ (Proposition~\ref{prop:model}), whose action is locally finite and locally unipotent (Theorem~\ref{thm:unip}). These properties allow us to compare lattice cohomology with Lie-algebra cohomology and then use algebraic induction to obtain
\[
H^q\bigl(\Gamma,\Omega^p_{\mathrm{pol}}(\C^n)\bigr)
\cong H^{p,q}_{\dbar}(\g,J),
\]
see Theorem~\ref{thm:E2}. Together with the analytic surjectivity established above and the known injectivity of $\iota$, this proves Theorem~\ref{thm:main}.

The Cartan--Leray identification is classical and is available for any compact complex manifold with Stein universal cover. However, nilpotency of $G$ and the polynomial realization of the universal cover as $\C^n$ are decisive in both parts of the proof. The Baker--Campbell--Hausdorff formula has finitely many terms and, together with the polynomial uniformization of \cite{HRSTW26}, provides the polynomial control underlying the growth estimates. In addition, $G_{\C}$ is a unipotent algebraic group, whose rational action on polynomial holomorphic forms is locally finite and locally unipotent, supplying the algebraic structure needed for the cohomology comparison. The uniformization theorem, the polynomial growth estimates, and the locally unipotent algebraic action thus work together to make the reduction to invariant cohomology possible. A Stein universal cover, even one biholomorphic to $\C^n$ with deck transformations acting by affine maps, is not sufficient. For the solvable group $S=\C\ltimes_\phi\C^2$ with $\phi(x+\sqrt{-1}y) = \operatorname{diag} (e^x,e^{-x})$ and a lattice $\Lambda\subset S$, the universal cover of $\Lambda\backslash S$ is $\C^3$, yet invariant forms compute Dolbeault cohomology only for some lattices, and the Hodge numbers depend on $\Lambda$ \cite[Section~5.1]{Kas13}. There, the deck action is not locally unipotent, and $\Lambda$-invariant holomorphic forms may grow exponentially.
\medskip

Theorem~\ref{thm:main} has consequences beyond Dolbeault cohomology. By complex conjugation, the inclusion of invariant forms also induces isomorphisms in $\partial$-cohomology, and is therefore an $E_1$-isomorphism of double complexes (Theorem~\ref{thm:E1iso}), as defined in \cite{Ste21}. Consequently, every invariant of the Dolbeault double complex preserved by $E_1$-isomorphisms is computed by invariant forms (Corollary~\ref{cor:zigzag}). This includes Bott--Chern and Aeppli cohomologies, whose dimensions, together with the Betti numbers, characterize the validity of the $\partial\dbar$-lemma \cite{AT13}, and every page of the Fr\"olicher spectral sequence from $E_1$ onward, together with its differentials. For a fixed left-invariant complex structure $J$ on $G$, all these invariants are therefore independent of the lattice, and the rational structure induced by $\Gamma$, which played an important role in earlier approaches, does not affect them.

For Bott--Chern cohomology, Angella had already proved that the de Rham and Dolbeault isomorphisms imply the corresponding Bott--Chern isomorphism \cite[Theorem~3.7]{Ang13}. The Aeppli isomorphism then follows by Hodge-$*$ duality with respect to an invariant Hermitian metric \cite[Section~2.2]{Ang13}. Theorem~\ref{thm:main} makes this implication applicable to every complex nilmanifold and, together with Nomizu's theorem, settles the following conjecture of Angella \cite[Conjecture~3.10]{Ang13}, which we restate in our notation.

\begin{conjecture*}[Angella]
Let $(M=\Gamma\backslash G, J)$ be a complex nilmanifold. Then the de Rham, Dolbeault, Bott--Chern, and Aeppli cohomologies of $M$ are computed by the corresponding subcomplexes of left-invariant forms.
\end{conjecture*}

Theorem~\ref{thm:main} also settles the following conjecture of Hasegawa on small deformations \cite{Has10}.

\begin{conjecture*}[Hasegawa]
Small deformations of left-invariant complex structures on a complex nilmanifold are all left-invariant.
\end{conjecture*}

The Iwasawa manifold, which is a complex parallelizable nilmanifold, illustrates the subtleties of deformation theory in the non-K\"ahler setting, since its small deformations need not be complex parallelizable, and their Hodge numbers can jump \cite{Nak75}.

Rollenske proved that small deformations of $J$ are left-invariant whenever invariant forms compute the Dolbeault cohomology of $(M,J)$ \cite[Theorem~2.6]{RolJLMS}. Theorem~\ref{thm:main} removes this assumption. To describe the resulting deformations explicitly, we use contraction with a nowhere-vanishing invariant holomorphic volume form \cite[Theorem~2.7]{BDV09} to transfer the cohomology isomorphism to Dolbeault cohomology with values in the holomorphic tangent bundle. Hence the inclusion of the invariant Kodaira--Spencer differential graded Lie algebra into the full Kodaira--Spencer algebra is a quasi-isomorphism (Lemma~\ref{thm:tangent}), and the Kuranishi construction takes place in the finite-dimensional invariant algebra $\Lambda^\bullet(\g^{0,1})^*\otimes\g^{1,0}$. Transporting the resulting left-invariant complex structures through the construction of Section~\ref{sec:setup} gives a uniform polynomial realization of the nearby fibers.

\begin{theorem}[see Theorem~\ref{thm:hasegawa}]
\label{thm:intro:deformations}
Every sufficiently small deformation of a complex nilmanifold $(M=\Gamma\backslash G,J)$ is biholomorphic to $(\Gamma\backslash G,J_t)$ for a left-invariant complex structure $J_t$ on $G$. More precisely, the nearby fibers of the Kuranishi family of $M$ admit presentations
\[
M_t\cong\Gamma_t\backslash\C^n,\qquad\Gamma_t\cong\Gamma,
\]
where $\Gamma_t$ acts freely, properly discontinuously and cocompactly on $\C^n$ by polynomial automorphisms of degree at most $s^{2n}$, with $n=\dim_{\C}M$ and $s$ the nilpotency step of $G$.
\end{theorem}

By completeness of the Kuranishi family, being a complex nilmanifold is therefore an open condition in holomorphic families of compact complex manifolds. Moreover, Theorem~\ref{thm:main} and Corollary~\ref{cor:zigzag} apply to every nearby fiber. Hence the Hodge, Bott--Chern and Aeppli numbers of all small deformations are computed by the invariant complexes of $(\g,J_t)$, and their jumps reduce to rank changes of finite-dimensional linear maps.

\subsection*{Acknowledgments}
The authors are grateful to Anna Fino, Hisashi Kasuya and Riccardo
Piovani for several useful comments on a first preprint of this paper
that improved its quality and presentation.

\section{Preliminaries and notation}\label{sec:prelim}

\subsection{Nilmanifolds and invariant complex structures.}\label{sec:nilm}
Let $M=\Gamma \backslash G$ be a nilmanifold with left-invariant complex structure $J$, where $G$ is a simply connected nilpotent Lie group with lattice $\Gamma$. Write $\g=\operatorname{Lie}(G)$ and denote by $L_g$ left translation by $g\in G$. The nilpotency step of $\g$ is denoted by $s$, so that
\[
(\operatorname{ad}X_1)\cdots(\operatorname{ad}X_s)=0
\qquad\text{for all }X_1,\dots,X_s\in\g.
\]

We use exponential coordinates on $G$, in which $\exp\colon\g\to G$ is a diffeomorphism and the group law is polynomial by nilpotency, and the Baker--Campbell--Hausdorff formula (BCH), which,  for $X$, $Y \in \g$, is given by 
\[
\log(\exp X \exp Y) = X + Y + \frac{1}{2}[X, Y]
+ \frac{1}{12}\left( [X, [X, Y]] + [Y, [Y, X]] \right) 
+ \ldots
\]
The series terminates since every bracket monomial of length greater than $s$ vanishes. In particular,  $\mathrm{BCH}$ is a polynomial map $\g\times\g\to\g$, of degree at most $s$, and it extends by the same formula to a polynomial map $\g_\C\times\g_\C\to\g_\C$.

Denote by $\g_\C=\g\otimes_\R\C$ the complexification of $\g$, and by $\g_\C=\g^{1,0}\oplus\g^{0,1}$ its $\pm i$-eigenspace decomposition under $J$, with projections $\pr^{1,0}$ and $\pr^{0,1}$. Integrability of $J$ means that $\g^{1,0}$ and $\g^{0,1}$ are complex subalgebras. Set $n \coloneqq \dim_\C\g^{1,0}$, so $\dim_\R\g = 2n$. We fix a basis $Z_1,\dots,Z_n$ of $\g^{1,0}$ and denote by $\zeta^1,\dots,\zeta^n$ the dual coframe of left-invariant $(1,0)$-forms, so that $\zeta^a(Z_b)=\delta_{ab}$ and $\zeta^a(\bar Z_b)=0$. We use the abbreviations
\[
\zeta^{ab}=\zeta^a\wedge\zeta^b,\qquad \zeta^{a\bar b}=\zeta^a\wedge\bar\zeta^b,\qquad \text{and}\quad \bar\zeta^{ab}=\bar\zeta^a\wedge\bar\zeta^b.
\]

\subsection{Invariant Dolbeault cohomology.}\label{sec:invdolbeault}
The left-invariant complex forms on $(G,J)$ of type $(p,q)$ are
\[
\Lambda^{p,q}\g_\C^*\ \coloneqq\ \Lambda^p(\g^{1,0})^*\otimes\Lambda^q(\g^{0,1})^*,
\]
and $\Lambda^\bullet\g_\C^*=\bigoplus_{p,q}\Lambda^{p,q}\g_\C^*$. The exterior derivative preserves left-invariant forms, and on them it is the Chevalley--Eilenberg differential of $\g_\C$. Integrability of $J$ is equivalent to $d\big(\Lambda^{1,0}\g_\C^*\big)\subseteq\Lambda^{2,0}\g_\C^*\oplus\Lambda^{1,1}\g_\C^*$. Consequently $d=\partial+\dbar$ on $\Lambda^\bullet\g_\C^*$ with
\[
\partial\colon\Lambda^{p,q}\g_\C^*\to\Lambda^{p+1,q}\g_\C^*,\qquad \dbar\colon\Lambda^{p,q}\g_\C^*\to\Lambda^{p,q+1}\g_\C^*,\qquad \dbar^2=0 .
\]
The \emph{invariant Dolbeault cohomology} of $(\g,J)$ is
\[
H^{p,q}_{\dbar}(\g,J)\ \coloneqq\ \frac{\ker\big(\dbar\colon\Lambda^{p,q}\g_\C^*\to\Lambda^{p,q+1}\g_\C^*\big)}{\dbar\,\Lambda^{p,q-1}\g_\C^*} \quad \text{with} \quad h^{p,q}(\g,J)\coloneqq\dim_\C H^{p,q}_{\dbar}(\g,J).
\]
Left-invariant forms descend to $M$, and pullback by the covering projection $\pi\colon G\to M$ commutes with $\dbar$, so there is a natural map
\[
\iota\colon H^{p,q}_{\dbar}(\g,J)\longrightarrow H^{p,q}_{\dbar}(M),
\]
the object of Theorem~\ref{thm:main}. We write $h^{p,q}(M)\coloneqq\dim_\C H^{p,q}_{\dbar}(M)$,  which is finite because $M$ is compact.

The algebra $\g^{0,1}$ acts on $\g_\C/\g^{0,1}\cong\g^{1,0}$ by $\bar X\cdot Y=\pr^{1,0}[\bar X,Y]$. This is the isotropy representation of the pair $(\g_\C,\g^{0,1})$, and it is a representation precisely because $\g^{0,1}$ is a subalgebra. Dualizing and taking exterior powers gives a $\g^{0,1}$-module $\m \coloneqq \Lambda^p(\g^{1,0})^*$. We use the identification
\begin{equation}\label{eq:CEidentification}
\Lambda^{p,q}\g_\C^*\ =\ \Lambda^q(\g^{0,1})^*\otimes \m \ =\ C^q(\g^{0,1}, \m ),
\end{equation}
Here $C^q(\g^{0,1},\m)=\operatorname{Hom}_\C(\Lambda^q\g^{0,1},\m)$ is the space of Chevalley--Eilenberg $q$-cochains with values in $\m$. The identification sends $\beta\otimes\lambda$ to $\beta\wedge\lambda=(-1)^{pq}\lambda\wedge\beta$, for $\beta\in\Lambda^q(\g^{0,1})^*$ and $\lambda\in\m$. With this convention, $\dbar$ corresponds to the Chevalley--Eilenberg differential, as verified in Theorem~\ref{thm:chain}~(d).

\subsection{Group cohomology and rational modules.}\label{sec:groupcoh}\label{sec:rational}
Group cohomology is taken in the abstract (discrete) sense. We consider $\C$-vector spaces $V$ with a $\C$-linear $\Gamma$-action, and $H^q(\Gamma,V) = \mathrm{Ext}^q_{\Z\Gamma}(\Z,V)$. We write $V^\Gamma$ for the subspace of vectors fixed by every element of $\Gamma$. We compute it with the homogeneous bar complex, in which $C^q(\Gamma,V)$ is the space of $\Gamma$-equivariant maps $f\colon\Gamma^{q+1}\to V$, that is $f(\gamma\gamma_0,\dots,\gamma\gamma_q)=\gamma\cdot f(\gamma_0,\dots,\gamma_q)$, with coboundary
\[
(\delta_\Gamma f)(\gamma_0,\dots,\gamma_{q+1})=\sum_{k=0}^{q+1}(-1)^k\,f(\gamma_0,\dots,\widehat{\gamma_k},\dots,\gamma_{q+1}),
\]
so that $H^0(\Gamma,V)=V^\Gamma$. In the inhomogeneous normalization, this corresponds to
\[
\bar f(\gamma_1,\dots,\gamma_q) \coloneqq f(e,\gamma_1,\gamma_1\gamma_2,\dots,\gamma_1\cdots\gamma_q).
\]
No topology on $V$ is used: in the inhomogeneous normalization, the cochains are arbitrary set maps $\Gamma^q\to V$. This is essential in Section~\textup{\ref{sec:proof:E1}}, where a $\dbar$-primitive is chosen separately for each argument $\vec\gamma\in\Gamma^i$, with no control whatsoever in $\vec\gamma$.
\medskip

Let $H$ be a linear algebraic group over $\C$. A \emph{rational $H$-module} is a $\C$-vector space $V$ with a linear $H$-action such that $V$ is the union of its finite-dimensional $H$-stable subspaces $V'$, on each of which $H$ acts through a morphism of algebraic groups $H\to\GL(V')$. It is \emph{locally unipotent} if each such action is unipotent, i.e., admits an $H$-stable complete flag with trivial one-dimensional subquotients. We write $V^H$ for the subspace of vectors fixed by every element of $H$. The \emph{rational cohomology} $\Hrat^\bullet(H,-)$ is the derived functor of $V\mapsto V^H$ on rational $H$-modules, computed by the Hochschild complex $C^t(H,V)=\C[H]^{\otimes t}\otimes V$. For a closed subgroup $K\subseteq H$ and a rational $K$-module $W$, the \emph{algebraically induced} module is
\[
\ind_K^H W\ \coloneqq\ \big(\C[H]\otimes W\big)^K\ =\ \{\,f\colon H\to W \text{ regular}\ :\ f(hk)=k^{-1}\!\cdot f(h)\ \ \forall h\in H,\ k\in K\,\}.
\]
Here a regular map $f\colon H\to W$ means a finite sum $f(h)=\sum_{j=1}^r a_j(h)w_j$, with $a_j\in\C[H]$ and $w_j\in W$. The $K$-action on $\C[H]\otimes W$ is given by $(k\cdot f)(h)=k\cdot f(hk)$. Thus $K$-invariance is equivalent to $f(hk)=k^{-1}\cdot f(h)$. The induced module is a rational $H$-module for $(x\cdot f)(h)=f(x^{-1}h)$. When $W$ is finite-dimensional, this is the space of regular sections of the associated algebraic vector bundle $H\times_K W\to H/K$. In particular, $\ind_K^H\C=\C[H/K]$. The functor $\ind_K^H$ is right adjoint to restriction, and Frobenius reciprocity reads $(\ind_K^H W)^H=W^K$. The category of rational $H$-modules is abelian and has enough injectives, so these derived functors exist. The hypotheses for the Grothendieck spectral sequence used below are verified in Theorem~\ref{thm:chain}~(c). 

\section{The polynomial model of the universal cover}\label{sec:setup}

In this section we show how the polynomial biholomorphism between $(G, J)$ and $\C^n$ constructed in \cite{HRSTW26} provides the main ingredients throughout the paper: the deck representation $\rho\colon G\to\Aut_{\mathrm{pol}}(\C^n)$, a uniform bound $D$ on the $w$-degree of its values, and the $\Gamma$-module $\Opoly$ of polynomial holomorphic forms.

\subsection{The universal cover of \texorpdfstring{$M$}{}.}

We recall \cite[Theorem~1.2]{HRSTW26}, which settles a conjecture by Hasegawa on the universal cover of nilmanifolds \cite{Has10}
\cite{Has09,Has10}.

\begin{theorem}\label{thm:hrstw}
Let $(G,J)$ be a simply connected nilpotent Lie group with a left-invariant complex structure. There exists a biholomorphism $\Phi\colon(G,J)\to\C^n$ such that $\Phi\circ\exp\colon\g\to\C^n$ is polynomial in real linear coordinates on $\g$.
\end{theorem}

Throughout, $\Phi$ denotes the map constructed in \cite[Sections~2 and~3]{HRSTW26}. We recall below its construction, due to Snow \cite{Sno86}. Polynomiality of $\log\circ\Phi^{-1}$ is proved in Lemma~\ref{lemma:phiinv}.

Since $G$ is simply connected and nilpotent, it embeds in its universal complexification $\GC$. The exponential $\exp\colon\g_\C\to\GC$ is a biholomorphism and the group law in exponential coordinates is $\mathrm{BCH}$, so that $\GC$ is an affine variety isomorphic to $\g_\C$ and $G^{1,0}\coloneqq\exp\g^{1,0}$, $\Gzo\coloneqq\exp\g^{0,1}$ are closed complex subgroups. Identify $\C^n$ with $\g^{1,0}$ through the basis $Z_1,\dots,Z_n$, set $X\coloneqq\GC/\Gzo$ with base point $o=e\Gzo$ and projection $\varpi\colon\GC\to X$, and let $\psi\colon X\to\C^n$ be the inverse of $w\mapsto\exp(w)\Gzo$, which is a bijection because $\GC=G^{1,0}\cdot\Gzo$ with $G^{1,0}\cap\Gzo=\{e\}$ \cite[Section 2]{HRSTW26}. Snow's map is the composition
\[
\Phi\colon\ G\ \lhook\joinrel\longrightarrow\ \GC\ \stackrel{\varpi}{\longrightarrow}\ X\ \stackrel{\psi}{\longrightarrow}\ \C^n .
\]

Explicitly, the multiplication map
\[
\Psi\colon\ \g^{1,0}\oplus\g^{0,1}\longrightarrow\g_\C,\qquad
\Psi(w,u)=\log\big(\exp(w)\exp(u)\big)=\mathrm{BCH}(w,u),
\]
is a polynomial automorphism of $\g_\C$, of degree at most $s$ \cite[Remark~2.2 and Section~3]{HRSTW26}. Writing $\Psi^{-1}=(P,Q)$, with $P\colon\g_\C\to\g^{1,0}$ and $Q\colon\g_\C\to\g^{0,1}$ polynomial, one has
\begin{equation}\label{eq:PQ}
\exp\zeta=\exp\big(P(\zeta)\big)\exp\big(Q(\zeta)\big),\qquad \zeta\in\g_\C ,
\end{equation}
hence $\psi\big(\exp(\zeta)\Gzo\big)=P(\zeta)$ and
\begin{equation}\label{eq:Phi}
\Phi(\exp\xi)=P(\xi),\qquad \xi\in\g .
\end{equation}
In particular $\Phi(e)=0$ and
\begin{equation}\label{eq:degP}
d\coloneqq\deg P\ \le\ \deg\Psi^{-1}\ \le\ (\deg\Psi)^{2n-1}\ \le\ s^{2n-1}
\end{equation}
by the degree bound for the inverse of a polynomial automorphism of $\C^{2n}$ \cite[Corollary~1.4]{BCW82}. Let us denote by $\Aut_{\mathrm{pol}}(\C^n)$ the group of polynomial automorphisms of $\C^n$, that is, bijections $\C^n\to\C^n$ whose components lie in $\C[w]=\C[w_1,\dots,w_n]$ and whose inverse has the same property. For $g\in G$, consider the biholomorphism of $\C^n$ given by $\rho(g)\coloneqq\Phi \circ L_g \circ \Phi^{-1}$. Then $\rho \colon G \to \mathrm{Bihol}(\C^n)$ is a homomorphism.

The quotient map $\pi\colon\C^n\to \rho(\Gamma) \backslash \C^n \cong M$ is the universal covering, with $\rho(\Gamma)$ acting freely and properly discontinuously. The Dolbeault cohomology of $M$ is computed from this covering by the standard Cartan--Leray spectral sequence, realized explicitly as a double complex in Section~\ref{sec:analytic}.

\begin{lemma}\label{lemma:phiinv}
The map $\log\circ\Phi^{-1}\colon\C^n\to\g$ is polynomial in $(w,\bar w)$. With $d=\deg P$, we have
\[
\deg(\Phi\circ\exp)\le d\le s^{2n-1},
\qquad
\deg(\log\circ\Phi^{-1})\le s^2d\le s^{2n+1}.
\]
These bounds are independent of $J$ for fixed $\g$.
\end{lemma}

\begin{proof}
Let $w\in\g^{1,0}\cong\C^n$, and set
$\xi=\log\Phi^{-1}(w)\in\g$. Write
$\Psi^{-1}(\xi)=(w,v)$, so that
$\exp\xi=\exp(w)\exp(v)$.
Since $\xi$ is real, inversion and complex conjugation give
\[
\exp(-\xi)=\exp(-\bar v)\exp(-\bar w).
\]
Taking inverses and comparing the two expressions for $\exp\xi$
yields
\[
\exp(-\bar w)\exp(w)=\exp(\bar v)\exp(-v).
\]
Here $\bar v\in\g^{1,0}$ and $-v\in\g^{0,1}$. Uniqueness of the factorization \eqref{eq:PQ} therefore gives
\[
P(\mathrm{BCH}(-\bar w,w))=\bar v.
\]
Consequently,
\[
\log\circ\Phi^{-1}(w)
=\mathrm{BCH}\bigl(\bar w,P(\mathrm{BCH}(-\bar w,w))\bigr),
\]
which is polynomial in $(w,\bar w)$.

Since $\Psi$ is a polynomial automorphism of $\C^{2n}$ of degree at most $s$, \cite[Corollary~1.4]{BCW82} gives
\[
d=\deg P\le\deg\Psi^{-1}\le(\deg\Psi)^{2n-1}\le s^{2n-1}.
\]
The inner BCH term in the displayed formula has degree at most $s$, its image under $P$ has degree at most $ds$, and the outer BCH
composition has degree at most $s^2d$. Thus
\[
\deg(\log\circ\Phi^{-1})\le s^2d\le s^{2n+1}.
\]
The forward bound follows from $\Phi\circ\exp=P|_\g$. The estimates depend only on $n$ and $s$.
\end{proof}

\begin{proposition}\label{prop:uniform:degree}
Let $\Phi$ be the biholomorphism of Theorem~\ref{thm:hrstw} and let $P$, $d$ be as in \eqref{eq:PQ} and \eqref{eq:degP}. Then
\[
F(\xi,w)\coloneqq P\big(\mathrm{BCH}(\xi,w)\big),\qquad \xi\in\g_\C,\ w\in\g^{1,0}\cong\C^n,
\]
is a polynomial map on $\g_\C\times\C^n$ with complex coefficients, and
\[
F(\log g,w)=\big(\Phi\circ L_g\circ\Phi^{-1}\big)(w),\qquad g\in G,\ w\in\C^n .
\]
Setting $D\coloneqq\deg_wF$ we have
\[
\deg_wF(\xi,\cdot)\le D \le s^{2n} \quad\text{for every }\xi\in\g_\C ,
\]
and in particular $\deg_w\big(\Phi\circ L_g\circ\Phi^{-1}\big)\le D$ for every $g\in G$. Moreover $D\ge1$.
\end{proposition}

\begin{proof}
Let $g\in G$ and $w\in\g^{1,0}$. By the definition of $\psi$, the point $\Phi^{-1}(w)$ lies in the coset $\exp(w) \, \Gzo$, so that
\[
g\,\Phi^{-1}(w)\in\exp\big(\mathrm{BCH}(\log g,w)\big)\Gzo ,
\]
and therefore, by \eqref{eq:PQ},
\[
\big(\Phi\circ L_g\circ\Phi^{-1}\big)(w)=\psi\Big(\exp\big(\mathrm{BCH}(\log g,w)\big)\Gzo\Big)=P\big(\mathrm{BCH}(\log g,w)\big).
\]
Since $\mathrm{BCH}$ is polynomial on $\g_\C\times\g_\C$ and $P$ is polynomial, $F$ is polynomial and is defined for every $\xi\in\g_\C$. Every $\mathrm{BCH}$ bracket monomial has length at most $s$, so $\deg_w\mathrm{BCH}(\xi,w)\le s$. With \eqref{eq:degP} this gives $D\le ds\le s^{2n}$. Specializing $\xi$ cannot raise the degree in $w$. Finally $\mathrm{BCH}(0,w)=w$ and $P(w)=w$ by \eqref{eq:PQ}, so $F(0,w)=w$ and $D\ge1$. A change of the basis of $\g^{1,0}$ replaces $\Phi$ by $A\circ\Phi$ with $A\in\GL(n,\C)$, which does not change $D$.
\end{proof}

In exponential coordinates, left translation by $\exp\xi$, with $\xi\in\g$, is the map $\eta\mapsto\mathrm{BCH}(\xi,\eta)$. Thus $F(\xi,\cdot)$ is this translation transported through $\Phi$. More generally, $F$ is the coordinate expression of the algebraic $G_{\C}$-action on the affine homogeneous space $G_{\C}/G^{0,1}\cong\C^n$, described in Proposition~\ref{prop:model}(a). Consequently,
\[
F(\mathrm{BCH}(\xi,\eta),w)=F(\xi,F(\eta,w)),
\qquad \xi,\eta\in\g_{\C}.
\]
Its restriction to $\Gamma$ is the deck action $\rho(\gamma)=F(\log\gamma,\cdot)$.

\begin{corollary}\label{lemma:rho:pol}
For every $g\in G$ one has $\rho(g)\in\Aut_{\mathrm{pol}}(\C^n)$, with inverse $\rho(g^{-1})$.
\end{corollary}

\begin{proof}
The map $\rho(g)=F(\log g,\cdot)$ is polynomial by Proposition~\ref{prop:uniform:degree}, and so is $\rho(g^{-1})$. Since $\rho$ is a homomorphism, $\rho(g)\circ\rho(g^{-1})=\rho(e)=\id$.
\end{proof}

\subsection{The module of polynomial forms.}\label{sec:module}
Let $\Opoly \coloneqq \Opoly(\C^n)$ be the space of holomorphic $p$-forms with polynomial coefficients. An element $\omega \in \Opoly$ can be written as
\[
\omega = \sum_{|I|=p} f_I\,dw^I,
\]
with $f_I\in\C[w]=\C[w_1,\dots,w_n]$ and $dw^I = dw^{i_1} \wedge \cdots \wedge dw^{i_p}$.

\begin{proposition}\label{prop:module}
$\Opoly$ is a left $\Gamma$-module under the action $\gamma \cdot \omega \coloneqq \rho (\gamma^{-1})^* \omega$.
\end{proposition}

\begin{proof}
For $\phi \in \Aut_{\mathrm{pol}}(\C^n)$ and $\omega=\sum_I f_I\,dw^I$, the action of $\phi$ is given by
\[
\phi^* \omega = \sum_I (f_I \circ \phi) \; d\phi^{i_1} \wedge \cdots \wedge d\phi^{i_p}, \quad \text{with } d\phi^i = \sum_j \tfrac{\partial\phi^i}{\partial w_j}\,dw^j.
\]
Because $\phi$ is holomorphic and polynomial, $\partial\phi/\partial w_j\in\C[w]$ and no $d\bar w$ terms appear, so that $\phi^*$ preserves type $(p,0)$ and polynomiality of $\omega$, mapping $\Opoly$ to itself. Pullback is contravariant, so that $\gamma\cdot\omega \coloneqq \rho(\gamma^{-1})^*\omega$ is a left action.
\end{proof}

We write $\sigma(g)\coloneqq\rho(g^{-1})^*$ for the operators of this action. On functions it acts as $(\sigma(g)f)(w)=f\big(\rho(g^{-1})w\big)$. The same convention is used for $\C[w]$ and for the modules of Section~\ref{sec:algebraic}. We denote by $\Omega^p(\C^n)$ the space of \emph{global} holomorphic $p$-forms on $\C^n$, that is, the space of sections $H^0(\C^n,\Omega^p_{\C^n})$ of the sheaf $\Omega^p_{\C^n}$ of holomorphic $p$-forms. $\Omega^p(\C^n)$ is a left $\Gamma$-module in the same way as $\Opoly$, with $\Opoly \subset \Omega^p(\C^n)$ as a $\Gamma$-submodule.
\medskip

\section{Polynomial surjectivity}\label{sec:analytic}

We show that, under the Cartan--Leray identification, every Dolbeault class of $M$ is represented by a group cocycle with values in $\Opoly$. More precisely, we prove the following.

\begin{theorem}\label{thm:E1}
The natural map
\[
H^q(\Gamma,\Opoly)\longrightarrow H^q\big(\Gamma,\Omega^p(\C^n)\big)
\]
induced by the inclusion $\Opoly\subset\Omega^p(\C^n)$ is surjective for all $(p,q)$. 
Since we have $H^q(\Gamma,\Omega^p(\C^n))\cong H^{p,q}_{\dbar}(M)$ by Proposition~\textup{\ref{prop:doublecx}},
it shows that every Dolbeault class of $M$ is the image of a class in $H^q(\Gamma,\Opoly)$.
\end{theorem}

The proof combines H\"ormander's weighted $L^2$-estimates with a Cartan--Leray double complex for the covering $\pi\colon\C^n\to M$. The former, combined with uniform interior estimates and Sobolev embedding, produces $\dbar$-primitives of polynomial growth. The latter identifies $H^{p,q}_{\dbar}(M)$ with the group cohomology $H^q(\Gamma,\Omega^p(\C^n))$ of holomorphic forms. 

\subsection{Polynomial growth $\dbar$-Poincar\'e lemma.}

We denote by $\partial^{(k)}$ derivatives in $w,\bar w$ of total order $k$,
acting coefficientwise.

A smooth form $\sigma$ on $\C^n$ has \emph{polynomial growth} if, for every $k\ge0$,
there are $C_k>0$ and $N_k\in\mathbb{N}$ such that
\[
\lvert\partial^{(k)}\sigma(w)\rvert\le C_k(1+|w|)^{N_k},\qquad w\in\C^n,
\]
where $|\,\cdot\,|$ denotes the standard pointwise norm on $\C^n$.
We write $\Apoly = \Apoly(\C^n)$ for the space of $(p,q)$-forms of polynomial growth.

The exponent $N_k$ is allowed to depend on $k$.

For an open set $\Omega\subseteq\C^n$ and a continuous real function
$\varphi$ on $\Omega$, we write $L^2(\Omega,\varphi)$ for the space of
functions or forms square-integrable with respect to $e^{-\varphi}\,dV$,
where $dV$ is Lebesgue measure. We write $L^2(\Omega)$ when $\varphi=0$.
We use the following $L^2$-estimate for the $\dbar$-equation
\cite[Lemma~4.4.1]{Hor90}.

\begin{theorem*}[H\"ormander]
Let $\Omega\subseteq\C^n$ be a pseudoconvex open set in $\C^n$ and $c$ be a positive continuous function on $\Omega$.
Assume that for $\varphi\in C^2(\Omega;\R)$,  the complex Hessian of $\varphi$ is bounded below by $c$ as a Hermitian form:
\begin{equation}\label{eq:levi}
\sum_{j,k=1}^n \frac{\partial^2\varphi}{\partial w_j\partial\bar w_k}(w)\,t_j\bar t_k\ \ge\ c(w)\,|t|^2,\qquad w\in\Omega,\ t\in\C^n.
\end{equation}
Then,  for every $\dbar$-closed $(p,q)$-form $\alpha \in L^2(\Omega, \varphi)$ on $\Omega$, $q\ge1$, 
with $\int_{\Omega}c^{-1}|\alpha|^2e^{-\varphi}\,dV<\infty$, there is a $(p,q-1)$-form $\beta \in L^2(\Omega, \varphi)$ 
with $\dbar \beta=\alpha$ and
\begin{equation}\label{eq:hormander}
\int_{\Omega}|\beta|^2e^{-\varphi}\,dV\ \le\ \int_{\Omega}c^{-1}|\alpha|^2e^{-\varphi}\,dV < \infty
\end{equation}
\end{theorem*}

\begin{lemma}\label{lemma:pg}
Let $\alpha$ be a smooth $\dbar$-closed $(p,q)$-form ($q\ge1$) on $\C^n$ of polynomial growth, with constants $C_k$ and 
exponents $N_k$ as above. Then there is a smooth $(p,q-1)$-form $\beta$ of polynomial growth with $\dbar\beta=\alpha$. 
To be precise, for every derivative of order $l\ge0$, we have
\[
\vert \partial^{(l)}\beta(w) \rvert \le K_l (1+|w|)^{M_l},
\]
with $K_l > 0$ and $M_l \in\mathbb{N}$. 
\end{lemma}

\begin{proof}
We write
\[
\alpha=\sum_{|I|=p} dw^I \wedge \alpha_I,
\]
with each $\alpha_I$ a $(0,q)$-form of polynomial growth. $\dbar$-closedness gives $\dbar \alpha_I=0$, and it suffices to 
solve $\dbar\beta_I=\alpha_I$ separately, then put $\beta = (-1)^p\sum_I dw^I\wedge\beta_I$, since 
$\dbar\big(dw^I\wedge\beta_I\big)=(-1)^p\,dw^I\wedge\dbar\beta_I$. Then $|\alpha|^2 = \sum_I|\alpha_I|^2$ and $|\beta|^2=\sum_I|\beta_I|^2$, 
so we obtain the desired bound on $\beta$ starting from the bound on the $\beta_I$. Thus, we can assume $\alpha$ is a $(0,q)$-form.

We apply H\"ormander's Theorem stated above with $\Omega=\C^n$ and the weight
\[
\varphi=(N_0+n+3) \log(1+|w|^2)\in C^\infty(\C^n).
\]
Its complex Hessian is
\[
\frac{\partial^2\varphi}{\partial w_j\partial\bar w_k}=(N_0+n+3) \left(\frac{\delta_{jk}}{1+|w|^2}-\frac{\bar w_j w_k}{(1+|w|^2)^2}\right),
\]
and can be bounded below by
\[
\sum_{j,k}\frac{\partial^2\varphi}{\partial w_j\partial\bar w_k}t_j\bar t_k
= (N_0+n+3) \left(\frac{|t|^2}{1+|w|^2}-\frac{|\langle t,w\rangle|^2}{(1+|w|^2)^2}\right)
\ge \frac{(N_0+n+3)\, |t|^2}{(1+|w|^2)^2}
\]
Hence \eqref{eq:levi} holds with
\[
c(w)=\frac{(N_0+n+3)}{(1+|w|^2)^2}>0.
\]
The right-hand side of \eqref{eq:hormander} is
\[
\int_{\C^n}c^{-1}|\alpha|^2e^{-\varphi}\,dV\ = \frac{1}{(N_0+n+3)} \int_{\C^n}(1+|w|^2)^{2}\,|\alpha|^2\,(1+|w|^2)^{-(N_0+n+3)}\,dV.
\]
From $|\alpha(w)|\le {C_0} (1+|w|)^{N_0}$ and $(1+|w|)^2\le 2(1+|w|^2)$, we have that $|\alpha|^2\le 2^{N_0}{C_0}^2(1+|w|^2)^{N_0}$, and we can estimate
\[
\int_{\C^n}c^{-1}|\alpha|^2e^{-\varphi}\,dV\ \le \frac{2^{N_0} {C_0}^2}{(N_0+n+3)} \int_{\C^n}\frac{dV}{(1+|w|^2)^{n+1}}\ <\ \infty,
\]
since $\kappa= \int_{\C^n}(1+|w|^2)^{-(n+1)}\,dV<\infty$ converges in real dimension $2n$.
Since $c\le N_0+n+3$, this also implies $\alpha\in L^2(\C^n,\varphi)$.

\medskip

H\"ormander's estimate provides a solution $\beta$ of $\dbar\beta=\alpha$ with
\[
\int_{\C^n}|\beta|^2(1+|w|^2)^{-(N_0+n+3)}\,dV\ \le\ \frac{2^{N_0}\kappa \,{C_0}^2}{(N_0+n+3)}.
\]

Choose the solution $\beta$ of minimal norm in $L^2(\C^n,\varphi)$.
Orthogonal projection onto $(\ker\dbar)^\perp$ preserves the equation
and does not increase the norm, so the preceding estimate still holds.
For $q\geq 2$, every compactly supported smooth $(0,q-2)$-form
$u$ satisfies $\dbar u\in\ker\dbar$, and hence
\[
\langle\beta,\dbar u\rangle_\varphi=0.
\]
Thus $\dbar_\varphi^*\beta=0$ in the sense of distributions,
where $\dbar_\varphi^*$ is the weighted adjoint of $\dbar$;
for $q=1$ this equation is automatic by degree. Consequently,
\[
(\dbar+\dbar_\varphi^*)\beta=\alpha.
\]
The operator $\dbar+\dbar_\varphi^*$ on $(0,\bullet)$-forms
is elliptic with smooth coefficients. Interior elliptic regularity
therefore implies that $\beta$ is smooth.

Let $B(w,r)\coloneqq\{z\in\C^n:|z-w|<r\}$ for $r>0$.
For $z\in B(w,1)$, we have
\[
1+|z|^2\le(1+|z|)^2\le4(1+|w|)^2.
\]
The preceding weighted estimate therefore gives
\[
\|\beta\|_{L^2(B(w,1))}
\le C(n,N_0)\,C_0(1+|w|)^{N_0+n+3},
\]
where the constant is independent of $w$.
\medskip

All derivatives of $\varphi=(N_0+n+3)\log(1+|w|^2)$ of positive order are bounded.
Thus $\dbar+\dbar_\varphi^*$ has constant principal coefficients and uniformly bounded
lower-order coefficients, together with all their derivatives.
Interior elliptic estimates on translated balls give, for every integer $m\ge0$,
\[
\|\beta\|_{H^{m+1}(B(w,1/2))}
\le C_m\left(\|\alpha\|_{H^m(B(w,1))}
+\|\beta\|_{L^2(B(w,1))}\right),
\]
where $H^m$ denotes the usual unweighted $L^2$-Sobolev space and $C_m$ is independent of $w$.
Since $\C^n$ has real dimension $2n$, Sobolev embedding applies whenever
$m+1>l+n$. In particular, for integers $m\ge l+n$,
\[
|\partial^{(l)}\beta(w)|
\le C_{l,m}\left(
\sum_{k=0}^{m}\sup_{B(w,1)}|\partial^{(k)}\alpha|
+\|\beta\|_{L^2(B(w,1))}\right),
\]
with $C_{l,m}$ independent of $w$.
For $z\in B(w,1)$, the polynomial-growth bounds give
$|\partial^{(k)}\alpha(z)|\le 2^{N_k}C_k(1+|w|)^{N_k}$.
Taking $m=l+n$ and using the local $L^2$-bound above, we obtain
\[
|\partial^{(l)}\beta(w)|\le K_l(1+|w|)^{M_l},\qquad
M_l=\max\{N_0+n+3,N_1,\ldots,N_{l+n}\},
\]
where $K_l$ is independent of $w$. Hence $\beta$ has polynomial growth.

\end{proof}

\begin{lemma}\label{lemma:liouville}
A holomorphic $(p,0)$-form on $\C^n$ is of polynomial growth if and only if it is polynomial.
\end{lemma}

\begin{proof}
Let $\omega=\sum_{|I|=p}f_I\,dw^I$ be a holomorphic $(p,0)$-form of polynomial growth.
Each $f_I$ is entire on $\C^n$.

Fix $I$ and let $|f_I(w)|\le C(1+|w|)^N$. Expand $f_I=\sum_\alpha c_\alpha w^\alpha$, the
series converging on all of $\C^n$. Cauchy's estimates on the polydisc
$\{|w_j| \le R\}$ give
\[
|c_\alpha|\ \le\ R^{-|\alpha|}\sup_{|w_j|=R}|f_I|\ \le\ C\,R^{-|\alpha|}\big(1+\sqrt n\,R\big)^{N},
\]
since $|w|\le\sqrt n\,R$ on that torus. For $|\alpha|>N$ the right-hand side is
$O\big(R^{N-|\alpha|}\big)\to0$ as $R\to\infty$, so $c_\alpha=0$. Hence each $f_I$ is a
polynomial, of degree at most $N$.
\end{proof}

\begin{lemma}\label{lemma:frame}
A $\Gamma$-invariant smooth $(p,q)$-form on $\C^n$ lies in $\Apoly$.
\end{lemma}

\begin{proof}
Let $\zeta^1,\ldots,\zeta^n$ be a left-invariant $(1,0)$-coframe on $G$. Their push-forwards $\Phi_*\zeta^1,\ldots,\Phi_*\zeta^n$ form a $\rho(G)$-invariant $(1,0)$-coframe on $\C^n$, whose conjugate is $\Phi_*\bar\zeta^1,\ldots,\Phi_*\bar\zeta^n$. In exponential coordinates, the invariant frame and its dual have polynomial coefficients, obtained by differentiating the finite sum giving the BCH formula $y\mapsto\mathrm{BCH}(x,y)$ at $y=0$ for the frame, and $z\mapsto\mathrm{BCH}(-x,z)$ at $z=x$ for the dual coframe. Together with the polynomiality of $\Phi\circ\exp$ and $\log\circ\Phi^{-1}$ (Theorem~\ref{thm:hrstw} and Lemma~\ref{lemma:phiinv}), this shows that the changes between $\{\Phi_*\zeta^j,\Phi_*\bar\zeta^j\}$ and $\{dw^k,d\bar w^k\}$ have polynomial coefficients in $(w,\bar w)$ in both directions.

A $\Gamma$-invariant form descends to the compact $M$, so its coefficients 
in the invariant coframe $\{\Phi_*\zeta^I\wedge\Phi_*\bar\zeta^J\}$ are bounded, together with all 
their derivatives, along invariant vector fields. Rewriting in $\{dw^I\wedge d\bar w^J\}$ multiplies 
by polynomial factors, giving polynomial growth of every coefficient. 

The same holds for every derivative. Indeed, write
\[
\partial_{w_k}=\sum_j\big(a_{kj}\,\Phi_*Z_j+b_{kj}\,\Phi_*\bar Z_j\big),
\]
where $Z_1,\dots,Z_n$ is the left-invariant frame dual to $\zeta^1,\dots,\zeta^n$ and the $a_{kj},b_{kj}$ are
the values of the transported coframe on $\partial_{w_k}$, hence polynomial in $(w,\bar w)$. Similarly for $\partial_{\bar w_k}$.
By induction on $m$, $\partial^{(m)}$ applied to a coefficient function in the invariant coframe is a combination, with polynomial 
coefficients, of iterated $\Phi_*Z$- and $\Phi_*\bar Z$-derivatives of that function: the inductive step
differentiates a product, and differentiating the $a_{kj},b_{kj}$ again produces polynomials. 
Each iterated invariant derivative descends to $M$ and is therefore bounded, so every $\partial^{(m)}$ of
such a coefficient has polynomial growth. Rewriting in $\{dw^I\wedge d\bar w^J\}$ and applying
Leibniz's rule gives the same conclusion for every coordinate coefficient of the form.
\end{proof}

\subsection{$\Gamma$-acyclicity of smooth forms.}
\begin{lemma}\label{lemma:acyclic}
Let $\Gamma$ act freely, properly discontinuously and cocompactly on $\C^n$ by biholomorphisms. 
Then for each $(p,j)$, the $\Gamma$-module $A^{p,j}(\C^n)$ of smooth forms is acyclic 
($H^{i}(\Gamma,\allowbreak A^{p,j}(\C^n))=0$ for $i>0$), and $H^0(\Gamma, A^{p,j}(\C^n))= A^{p,j}(M)$.
\end{lemma}

\begin{proof}
Cocompactness gives a compact $K\subset\C^n$ whose $\Gamma$-translates cover $\C^n$. 
Choose a nonnegative bump function $\psi\in C^\infty_c(\C^n)$ with $\psi\equiv1$ on $K$, 
and set $\chi \coloneqq \psi / ( \sum_{\gamma\in\Gamma} \gamma \cdot \psi)$. The sum is 
locally finite by proper discontinuity, and positive on $\C^n$. Then $\sum_{\gamma}\gamma\cdot\chi=1$.
On homogeneous cochains define
\[
(hf)(\gamma_0,\dots,\gamma_{i-1})\coloneqq\sum_{\gamma\in\Gamma}(\gamma\cdot\chi)\,f(\gamma,\gamma_0,\dots,\gamma_{i-1}),
\]
a locally finite sum. Multiplication by the compactly supported $\gamma \cdot \chi$ preserves $A^{p,j}$. 
The cochain $hf$ is again equivariant since for $\eta\in\Gamma$ we have
\begin{align*}
(hf)(\eta\gamma_0,\dots,\eta\gamma_{i-1}) &=\sum_{\gamma}(\gamma\cdot\chi)\,f(\gamma,\eta\gamma_0,\dots,\eta\gamma_{i-1}) \\
&=\sum_{\gamma}(\eta\gamma\cdot\chi)\,f(\eta\gamma,\eta\gamma_0,\dots,\eta\gamma_{i-1}) \\
&=\eta\cdot (hf)(\gamma_0,\dots,\gamma_{i-1}),
\end{align*}
where we reindexed $\gamma\mapsto\eta\gamma$ and used equivariance of $f$. For $i\ge1$, we get
\[
(\delta_\Gamma hf+h\delta_\Gamma f)(\gamma_0,\dots,\gamma_i)=\sum_{\gamma}(\gamma\cdot\chi)\,f(\gamma_0,\dots,\gamma_i)=f(\gamma_0,\dots,\gamma_i),
\]
since all the remaining terms of $\delta_\Gamma hf$ cancel against those of $h\delta_\Gamma f$ and $\sum_\gamma\gamma\cdot\chi=1$.
Hence higher cohomology vanishes. Finally, $H^0$ consists of $\Gamma$-invariant forms, i.e., forms descending to $M$.
\end{proof}

\subsection{The Cartan--Leray double complex.}
For each fixed $p$, consider the first-quadrant double complex of homogeneous group cochains valued in smooth forms
\[
K^{i,j}_p \coloneqq C^i\big(\Gamma, A^{p,j}(\C^n)\big),
\]
with horizontal differential the group coboundary $\delta_\Gamma$, and with vertical differential $\dbar$. 
We use the total differential $D_{\mathrm{tot}} \coloneqq\delta_\Gamma+(-1)^i\dbar$ on $K^{i,j}_p$,
so that $D_{\mathrm{tot}}^2=0$.

\begin{proposition}\label{prop:doublecx}
The natural inclusion of $\Gamma$-invariant forms into
the total complex induces an isomorphism
\[
H^{p,\bullet}_{\dbar}(M)
\cong H^\bullet_{\mathrm{tot}}(K_p).
\]
Explicitly:
\begin{enumerate}[label=\textup{(\roman*)}]
\item taking $\dbar$-cohomology first, the columns are exact for $j\ge1$ 
with $\ker(\dbar\colon A^{p,0}(\C^n)\to A^{p,1}(\C^n))=\Omega^p(\C^n)$, so the $j$-first spectral sequence 
degenerates to $H^i(\Gamma,\Omega^p(\C^n))$;

\item taking $\delta_\Gamma$-cohomology first, the rows are $\Gamma$-acyclic (Lemma~\ref{lemma:acyclic}),
so the $i$-first spectral sequence degenerates to $H^{p,\bullet}_{\dbar}(M)$.
\end{enumerate}

Both spectral sequences compute $H^\bullet_{\mathrm{tot}}(K_p)$. Comparing them gives a natural 
isomorphism $H^{p,q}_{\dbar}(M)\cong H^q(\Gamma,\Omega^p(\C^n))$ for all $(p,q)$.

\end{proposition}

\begin{proof}
$K^{\bullet,\bullet}_p$ is first-quadrant, so the two spectral sequences induced by the row and column 
filtrations both converge to $H^\bullet_{\mathrm{tot}}(K_p)$.

For (i), the complex $(A^{p,\bullet}(\C^n),\dbar)$ is exact in positive degrees, with kernel $\Omega^p(\C^n)$ 
in degree $0$. This is the vanishing $H^{p,j}_{\dbar}(\C^n)=0$ for $j\ge1$, which holds because $\C^n$ is Stein. 
Exactness passes to the columns: as vector spaces $C^i(\Gamma,V)\cong\prod_{\Gamma^i}V$, an equivariant homogeneous 
cochain being determined by arbitrary values on a set of representatives, and an arbitrary product of exact 
sequences of vector spaces is exact. Hence the $j$-first sequence has $E_1$ concentrated in the row $j=0$, 
where it reads $C^\bullet(\Gamma,\Omega^p(\C^n))$, and it degenerates to $H^i(\Gamma,\Omega^p(\C^n))$.

For (ii), the rows are $\Gamma$-acyclic by Lemma~\ref{lemma:acyclic}, so the $i$-first sequence has $E_1$ concentrated 
in the column $i=0$, where it reads $A^{p,\bullet}(M)$, and it degenerates to $H^{p,\bullet}_{\dbar}(M)$.
\end{proof}

\subsection{Proof of Theorem~\ref{thm:E1}.}\label{sec:proof:E1}
\begin{proof}[Proof of Theorem~\ref{thm:E1}]
Fix $[\alpha]\in H^{p,q}_{\dbar}(M)$ and choose any smooth representative $\alpha$. Its lift $\tilde\alpha$ to
$\C^n$ is $\Gamma$-invariant and $\dbar$-closed, hence lies in $\Apoly$ by Lemma~\ref{lemma:frame}. 
Throughout we use the identifications of Proposition~\ref{prop:doublecx}.

If $q=0$, then $\tilde\alpha$ is a $\Gamma$-invariant holomorphic $p$-form lying in $\Apoly[p,0]$, 
hence in $\Opoly$ by Lemma~\ref{lemma:liouville}. Since it is a $\Gamma$-invariant element of $\Opoly$ mapping to 
$[\alpha]$, there is nothing more to prove. 

Assume from now on $q\ge1$. We produce a $\Gamma$-$q$-cocycle valued in $\Opoly$ mapping to $[\alpha]$, 
by the standard zig-zag in $K^{\bullet,\bullet}_p$, kept inside the polynomial-growth class.

We compute in the inhomogeneous normalization of $K^{\bullet,\bullet}_p$. An inhomogeneous $i$-cochain is 
an arbitrary map $\Gamma^i\to A^{p,j}(\C^n)$. A $0$-cochain is a single form, and the coboundary is
\begin{align*}
(\delta_\Gamma f)(\gamma_1,\dots,\gamma_{i+1}) &=\gamma_1\cdot f(\gamma_2,\dots,\gamma_{i+1})\\
&+\sum_{k=1}^{i}(-1)^kf(\gamma_1,\dots,\gamma_k\gamma_{k+1},\dots,\gamma_{i+1})\\ 
&+(-1)^{i+1}f(\gamma_1,\dots,\gamma_i),
\end{align*}
so that $(\delta_\Gamma\beta)(\gamma)=\gamma\cdot\beta-\beta$ in degree $0$. The correspondence between inhomogeneous 
and homogeneous cochains of Section~\ref{sec:groupcoh} identifies these cochains with those of $K^{\bullet,\bullet}_p$, 
in a way compatible with $\dbar$ and the bigrading. Only one summand of $\delta_\Gamma f$ carries a module action,
a point used repeatedly below.

By Lemma~\ref{lemma:pg}, take $\beta_0 \in \Apoly[p,q-1]$ such that $\dbar\beta_0=\tilde\alpha$, 
and set $c^1 \coloneqq \delta_\Gamma \beta_0$. For each fixed $\gamma$, we have that
\[
(\delta_\Gamma\beta_0)(\gamma)=\gamma\cdot\beta_0-\beta_0=\rho(\gamma^{-1})^*\beta_0-\beta_0.
\]
By Proposition~\ref{prop:uniform:degree} the map $\rho(\gamma^{-1})$ is a polynomial automorphism of $\C^n$, 
so the pullback composes the coefficients of $\beta_0$ with a polynomial map and multiplies them by 
polynomial Jacobian factors. Each $\partial^{(k)}\big(\rho(\gamma^{-1})^*\beta_0\big)$ is then a finite sum of terms $(\partial^{(j)}\beta_0)\circ\rho(\gamma^{-1})$ multiplied by products of derivatives of $\rho(\gamma^{-1})$, 
all of polynomial growth. Hence $\rho(\gamma^{-1})^*\beta_0$ has polynomial growth, 
and $c^1\in C^1(\Gamma, \Apoly[p,q-1])$ is $\dbar$-closed, 
since $\dbar\delta_\Gamma\beta_0=\delta_\Gamma\dbar\beta_0=\delta_\Gamma\tilde\alpha=0$ 
(because $\tilde\alpha$ is invariant, so $\delta_\Gamma\tilde\alpha=0$).

If $q=1$, $c^1$ takes values in $\Apoly[p,0]$, and is $\dbar$-closed, hence polynomial by Lemma~\ref{lemma:liouville}. 
So we have that $c^1\in C^1(\Gamma,\Opoly)$, and the zig-zag stops with $c^q=c^1$. 

If $q \ge 2$, we solve the $\dbar$-equation for each value $c^1(\gamma)$, $\gamma \in \Gamma$, separately
by Lemma~\ref{lemma:pg}, obtaining $\beta_1\in C^1(\Gamma, \Apoly[p,q-2])$ with $\dbar\beta_1=c^1$. 
No uniformity in $\gamma$ is required. Indeed, by the convention of Section~\ref{sec:groupcoh}, 
the cochains of $K^{\bullet,\bullet}_p$ are arbitrary set maps $\Gamma^i\to A^{p,j}(\C^n)$, 
so the choice may be made independently for each argument.

For $1\le k\le q-2$, the inductive step is the same. Assume $\beta_k\in C^k(\Gamma,\Apoly[p,q-k-1])$ has been constructed 
with $\dbar\beta_k=c^k$, and set $c^{k+1}\coloneqq\delta_\Gamma\beta_k$. 
For a fixed $\vec\gamma=(\gamma_1,\dots,\gamma_{k+1})$ the value $c^{k+1}(\vec\gamma)$ is a 
finite alternating sum of values of $\beta_k$, exactly one of which is acted on by $\gamma_1\cdot(-)=\rho(\gamma_1^{-1})^*$. 
By the previous paragraph every summand has polynomial growth, hence so does $c^{k+1}(\vec\gamma)$. Moreover $\dbar c^{k+1}=\delta_\Gamma\dbar\beta_k=\delta_\Gamma c^k=\delta_\Gamma^2\beta_{k-1}=0$, so Lemma~\ref{lemma:pg}, 
applied to each $\vec\gamma$ separately, yields
$\beta_{k+1}\in C^{k+1}(\Gamma,\Apoly[p,q-k-2])$ with $\dbar\beta_{k+1}=c^{k+1}$. 
The growth degree of $c^{k+1}(\vec\gamma)$ depends on $k$ and on $\vec\gamma$.

After $q$ steps the staircase reaches
\[
c^q\coloneqq\delta_\Gamma\beta_{q-1}\in C^q(\Gamma, \Apoly[p,0]),\qquad \dbar c^q=0,
\]
so $c^q$ takes values in $\ker(\dbar \colon \Apoly[p,0] \to \Apoly[p,1])=\Opoly$ by Lemma~\ref{lemma:liouville}. 
In addition, $c^q$ is a group $q$-cocycle because $\delta_\Gamma c^q=\delta_\Gamma^2\beta_{q-1}=0$.

We compare the two classes in $H^q_{\mathrm{tot}}(K_p)$. Setting $c^0\coloneqq\tilde\alpha$, 
the construction gives $D_{\mathrm{tot}} \beta_k=c^{k+1}+(-1)^k c^k$ for $0\le k\le q-1$. 
Put $\varepsilon_0\coloneqq-1$ and $\varepsilon_k\coloneqq(-1)^{k+1}\varepsilon_{k-1}$ for $1\le k\le q-1$,
so that $\varepsilon_{k-1}+(-1)^k\varepsilon_k=0$ and the sum gives
\[
D_{\mathrm{tot}}\Big(\sum_{k=0}^{q-1}\varepsilon_k\beta_k\Big)=\varepsilon_{q-1}c^q+\varepsilon_0c^0=\varepsilon_{q-1}c^q-\tilde\alpha,
\]
that is,
\[
\tilde\alpha + D_{\mathrm{tot}} \Big(\sum_{k=0}^{q-1}\varepsilon_k\beta_k\Big)=\varepsilon_{q-1}\,c^q.
\]
Thus the total-complex class of $\tilde\alpha$ (viewed in the column $i=0$) equals $\varepsilon_{q-1}$ times 
that of $c^q$ (viewed in the row $j=0$). Under the identifications of Proposition~\ref{prop:doublecx}, 
the image of $\varepsilon_{q-1}\,[c^q]\in H^q(\Gamma,\Opoly)$ under
\[
H^q(\Gamma,\Opoly)\to H^q(\Gamma,\Omega^p(\C^n))\cong H^{p,q}_{\dbar}(M)
\]
is $[\alpha]$. Surjectivity follows.
\end{proof}

\begin{remark}\label{rem:notclaimed}
The proof does not require $\Gamma$-acyclicity of $\Apoly$.
The partition-of-unity homotopy of Lemma~\ref{lemma:acyclic}
need not preserve polynomial growth: although its sum is locally
finite, the cochain values have no uniform growth bounds as the
group arguments vary. Instead, the staircase is performed in
the double complex of all smooth cochains, using primitives of
polynomial growth to ensure that the final cocycle takes values
in $\Opoly$.
\end{remark}

\section{From group cohomology to invariant cohomology}\label{sec:algebraic}

In this section, we identify the group cohomology
$H^q(\Gamma,\Opoly)$ with the invariant Dolbeault
cohomology $H^{p,q}_{\dbar}(\g,J)$.

\begin{theorem}\label{thm:E2}
There is an isomorphism $H^q(\Gamma,\Opoly) \cong H^{p,q}_{\dbar} (\g, J)$. In particular, we have $\dim_\C H^q(\Gamma,\Opoly)=h^{p,q}(\g,J)$ for all $(p,q)$.
\end{theorem}

We use throughout the language of rational modules and algebraic induction of Section~\ref{sec:rational}. The proof of Theorem~\ref{thm:E2} consists of the chain of isomorphisms of cohomologies presented in Theorem~\ref{thm:chain}, which moves the coefficients from $\Gamma$-modules to $\g^{0,1}$-modules through the homogeneous model $X=\GC/\Gzo$ of Section \ref{sec:model}. The model contributes since $X$ is affine, $\C[w]=\C[X]$, and $\Opoly\cong\ind_{\Gzo}^{\GC}\m$ as rational $\GC$-modules (Proposition~\ref{prop:model}). The proof consists of three standard mechanisms applied in sequence. Their hypotheses follow from local unipotence (Theorem~\ref{thm:unip}), which for a unipotent $\GC$ acting on an affine variety is a formal consequence of the model. 
\begin{enumerate}
    \item For finite-dimensional unipotent coefficients, van Est's
comparison theorem identifies the cohomology of $\Gamma$ with
that of $\g$. Passing to a directed union gives the comparison
for $\Opoly$. This is the only step in this section that uses
the lattice.
    
    \item Hochschild's theorem replaces the algebraic group by its Lie algebra. 
    
    \item Frobenius reciprocity, in the form of a Grothendieck spectral sequence which degenerates because $X$ is affine, moves the coefficients from $\GC$ down to $\Gzo$, at which point the complex obtained is the invariant Dolbeault complex.
\end{enumerate}

\subsection{The homogeneous model.}\label{sec:model} We keep the notation of Section~\ref{sec:setup}. Here $\GC$ is the simply connected complex nilpotent group with Lie algebra $\g_\C$, $G^{1,0}=\exp\g^{1,0}$ and $\Gzo=\exp\g^{0,1}$ are closed complex subgroups, $X=\GC/\Gzo$ has base point $o=e\Gzo$ and projection $\varpi\colon\GC\to X$, and $\psi\colon X\to\C^n$ is the map of Section~\ref{sec:setup}, with $P$ as in \eqref{eq:PQ}.

\begin{proposition}\label{prop:model}
\leavevmode
\begin{enumerate}[label=\textup{(\alph*)}]
\item
The map $\psi\colon X\to\C^n$ is an isomorphism of complex affine varieties, with $\psi(o)=0$, $\psi^{-1}(w)=\exp(w)\Gzo$, $\psi(\exp(\zeta)\Gzo)=P(\zeta)$, and $\Phi=\psi\circ\varpi|_G$. Consequently
\[
\rho(g)=\psi\circ\lambda_g\circ\psi^{-1},\qquad g\in G,
\]
where $\lambda_g$ denotes left translation by $g$ on $X$, and $\rho$ extends to the algebraic left action
\[
\rho_\C(x)\coloneqq\psi\circ\lambda_x\circ\psi^{-1},\qquad x\in\GC,
\]
of $\GC$ on $\C^n$, which is given by $\rho_\C(\exp\xi)(w)=F(\xi,w)$ and has $w$-degree at most $D$. In particular, the deck action of $\Gamma$ is the restriction to $\Gamma\subset G\subset\GC$ of the algebraic left $\GC$-action on $X$.
\item Via $\psi$, we identify $\C[X]$ with $\C[w]$.
As rational $\GC$-modules, we have
\[
\C[w]\cong\ind_{\Gzo}^{\GC}\C \quad \text{and} \quad  \Opoly\cong\ind_{\Gzo}^{\GC}\Lambda^p(\g_\C/\g^{0,1})^*=\ind_{\Gzo}^{\GC}\m.
\]
\end{enumerate}
\end{proposition}

\begin{proof}
\textbf{Proof of (a).} In exponential coordinates, $G^{1,0}$, $\Gzo$ and $\GC$ are the affine spaces $\g^{1,0}$, $\g^{0,1}$ and $\g_\C$, and the multiplication $G^{1,0}\times\Gzo\to\GC$ is the map $\Psi$ of Section~\ref{sec:setup}. Since $\Psi$ is a polynomial automorphism of $\g_\C$, the multiplication is an isomorphism of complex affine varieties. Hence $X$ carries a structure of affine variety for which $\varpi|_{G^{1,0}}\colon G^{1,0}\to X$ is an isomorphism, $\varpi$ is a morphism, and the action of $\GC$ on $X$ by left translations is a morphism $\GC\times X\to X$. This is the quotient structure of $\GC/\Gzo$. By the universal property of the quotient \cite[Chapter~II, Section~6]{Bor91}, $\varpi$ factors through a bijective morphism $\GC/\Gzo\to X$, whose inverse is the composition of $(\varpi|_{G^{1,0}})^{-1}$ with the morphism $G^{1,0}\hookrightarrow\GC\to\GC/\Gzo$.

The map $\psi$ is the composition of $(\varpi|_{G^{1,0}})^{-1}$ with $(\exp|_{\g^{1,0}})^{-1}$ and with the linear coordinates given by $Z_1,\dots,Z_n$, hence it is an isomorphism of complex affine varieties. In other words, $\psi$ is defined by the commutative square
\begin{equation}\label{eq:psi}
\begin{tikzcd}[row sep=large, column sep=huge]
G^{1,0} \arrow[r, "\varpi|_{G^{1,0}}", "\cong"'] & X \arrow[d, "\psi"] \\
\g^{1,0} \arrow[u, "\exp"] \arrow[r, "\cong"', "{(Z_1,\dots,Z_n)}"] & \C^n
\end{tikzcd}
\end{equation}
all of whose arrows are isomorphisms of complex affine varieties. The formula $\psi(\exp(\zeta)\Gzo)=P(\zeta)$ is \eqref{eq:PQ}, and $\psi(o)=0$ follows by taking $\zeta=0$. The identity $\Phi=\psi\circ\varpi|_G$ follows from the construction of $\Phi$ recalled in Section~\ref{sec:setup}.

For every $x\in\GC$, left translation satisfies $\varpi\circ L_x=\lambda_x\circ\varpi$. Restricting to $g\in G$ and using $\Phi=\psi\circ\varpi|_G$, we obtain
\[
\Phi\circ L_g=(\psi\circ\lambda_g\circ\psi^{-1})\circ\Phi.
\]
Thus $\rho(g)=\Phi\circ L_g\circ\Phi^{-1}=\psi\circ\lambda_g\circ\psi^{-1}$, which is equivalent to the diagram
\begin{equation}\label{eq:modelsquare}
\begin{tikzcd}[row sep=large, column sep=large]
G \arrow[r, "L_g"] \arrow[d, "\varpi|_G"] \arrow[dd, bend right=50, "\Phi"'] & G \arrow[d, "\varpi|_G"'] \arrow[dd, bend left=50, "\Phi"] \\
X \arrow[r, "\lambda_g"] \arrow[d, "\psi"] & X \arrow[d, "\psi"'] \\
\C^n \arrow[r, "\rho(g)"] & \C^n
\end{tikzcd}
\end{equation}
commuting for every $g\in G$.
Transporting the action morphism $\GC\times X\to X$
through $\psi$ gives the algebraic left action
$\rho_\C(x)=\psi\circ\lambda_x\circ\psi^{-1}$,
which restricts to $\rho$ on $G$.
The formula
$\rho_\C(\exp\xi)(w)=P(\mathrm{BCH}(\xi,w))=F(\xi,w)$
and the degree bound follow from
Proposition~\ref{prop:uniform:degree}, whose computation
is valid for every $\xi\in\g_\C$.

\textbf{Proof of (b).}
Recall from Section~\ref{sec:invdolbeault} that $\m=\Lambda^p(\g^{1,0})^*$.
By (a), $\psi$ identifies $\C[X]$ with $\C[w]$.
Pullback along the quotient map $\varpi$ identifies
regular functions on $X$ with right $\Gzo$-invariant
regular functions on $\GC$. Hence
\[
\C[X]\cong\C[\GC]^{\Gzo}
=\ind_{\Gzo}^{\GC}\C,
\]
where the last equality is the definition of algebraic
induction with trivial coefficients.
Equivariance for the $\GC$-actions is the case $p=0$
of the argument below.

Similarly, $\psi$ identifies $\Opoly$ with the space of algebraic
$p$-forms on $X$. These are precisely the forms
$\sum_I f_I\,dw^I$ with $f_I\in\C[w]$.
Using $T_oX\cong\g_\C/\g^{0,1}$ and the isotropy action
described in Remark~\ref{rem:isotropy}, we identify
$\m$ with $\Lambda^pT_o^*X$, endowed with the exterior dual
of the isotropy representation.
For an algebraic $p$-form $\alpha$, we consider the associated regular map
\[
f_\alpha\colon\GC\longrightarrow\m,\qquad
f_\alpha(x)=(\lambda_x^*\alpha)_o.
\]
For $h\in\Gzo$, the isotropy action gives
\[
f_\alpha(xh)=h^{-1}\cdot f_\alpha(x),
\]
so $f_\alpha\in\ind_{\Gzo}^{\GC}\m$.
Conversely, such an equivariant regular map $f$ determines a form by
\[
\alpha_{xo}
=\big(d\lambda_{x^{-1}}|_{xo}\big)^*f(x).
\]
Equivariance makes this definition independent of the representative
$x$, and the algebraic product decomposition in (a) shows that
$\alpha$ is regular. These constructions are inverse to one another.

Finally, for $a\in\GC$, the action
$a\cdot\alpha=\lambda_{a^{-1}}^*\alpha$ satisfies
\[
f_{a\cdot\alpha}(x)=f_\alpha(a^{-1}x).
\]
Thus the identification
$\Opoly\cong\ind_{\Gzo}^{\GC}\m$
is an isomorphism of rational $\GC$-modules.
\end{proof}
\begin{remark}\label{rem:isotropy}
The differential $d\varpi_e$ induces an isomorphism
$\g_\C/\g^{0,1}\cong T_oX$.
For $h\in\Gzo$ and $y\in\GC$, one has
\[
\lambda_h(\varpi(y))=hy\Gzo=\varpi(hyh^{-1}),
\]
and therefore
\[
(d\lambda_h)_o\circ d\varpi_e
=d\varpi_e\circ\Ad(h).
\]
Thus, under this isomorphism, the isotropy action of $\Gzo$
on $T_oX$ is the action induced by $\Ad$ on
$\g_\C/\g^{0,1}$. This action is well defined because
$\Ad(h)\g^{0,1}=\g^{0,1}$ for $h\in\Gzo$.
In general, the resulting action on $\g^{1,0}$ is not
the restriction of $\Ad$, since $\Ad(h)$ need not
preserve $\g^{1,0}$.
\end{remark}

\subsection{Local finiteness and unipotence.}
The chain of isomorphisms in Theorem~\ref{thm:chain} passes through finite-dimensional coefficients and a colimit. This subsection provides both: an exhaustion of $\C[w]$ and $\Opoly$ by finite-dimensional $\GC$-submodules, and the unipotence of the $\GC$-action on each of them.

\begin{construction}[$G_\C$-stable exhaustion]\label{con:Wm}
By Proposition~\ref{prop:model}, the $\Gamma$-module $\C[w]=\C[X]$ carries the restriction of the algebraic left $\GC$-action on $X$, given by $x\cdot f=\rho_\C(x^{-1})^*f$ for $x\in\GC$. For $\gamma\in\Gamma$ this is the action $\gamma\cdot f=\rho(\gamma^{-1})^*f$ of Section~\ref{sec:module}. Since $\exp\colon\g_\C\to\GC$ is surjective, every $\rho_\C(x)$ is of the form $F(\xi,\cdot)$ with $\xi\in\g_\C$. By Proposition~\ref{prop:model}~(a), it has $w$-degree at most $D$, the bound of Proposition~\ref{prop:uniform:degree}. Write $\C[w]_{\le m}$ for the space of polynomials of total degree $\le m$ in $w_1,\dots,w_n$. For $m\ge0$ set
\[
W_m \coloneqq \mathrm{span}_\C\{x\cdot f : x\in\GC,\ f\in\C[w]_{\le m}\}\ \subseteq\ \C[w]_{\le Dm}.
\]
Each $W_m$ is $\GC$-stable, finite-dimensional, contains $\C[w]_{\le m}$, and $\bigcup_m W_m=\C[w]$. Similarly, let $W_m^{(p)}\subseteq\Opoly$ be the $\GC$-span of the $x\cdot\omega$ with $\omega=\sum_I f_I\,dw^I$ and $\deg f_I\le m$. The form $x\cdot\omega=\rho_\C(x^{-1})^*\omega$ is computed as in Proposition~\ref{prop:module}: its coefficients are the compositions $f_I\circ\rho_\C(x^{-1})$, of degree at most $Dm$, multiplied by $p\times p$ minors of the Jacobian of $\rho_\C(x^{-1})$, whose entries have $w$-degree at most $D-1$ (recall $D\ge1$ by Proposition~\ref{prop:uniform:degree}). Hence
\[
W_m^{(p)}\ \subseteq\ \Big\{\textstyle\sum_I f_I\,dw^I\ :\ \deg f_I\le Dm+p(D-1)\Big\},
\]
so $W_m^{(p)}$ is finite-dimensional and $\GC$-stable, it contains the forms of coefficient degree $\le m$, and $\bigcup_m W_m^{(p)}=\Opoly$. Finally, denote by
\[
\alpha_m\colon\GC\longrightarrow\GL(W_m),\qquad \alpha_m^{(p)}\colon\GC\longrightarrow\GL\big(W_m^{(p)}\big),\qquad x\longmapsto x\cdot(-),
\]
the representations of $\GC$ on these subspaces. In exponential coordinates, we have that $x=\exp\xi$. Since $F$ is polynomial in $(\xi,w)$, their matrix coefficients are polynomial in $\xi$. Hence $\alpha_m$ and $\alpha_m^{(p)}$ are morphisms of algebraic groups.
\end{construction}

\begin{theorem}\label{thm:unip}
$\C[w]$ and $\Opoly$ are locally finite, locally unipotent rational $\GC$-modules. Each is the directed union of the finite-dimensional $\GC$-submodules $W_m$, resp.\ $W_m^{(p)}$, and, on each of these, $\GC$ acts by a unipotent algebraic representation, with a $\GC$-stable complete flag whose one-dimensional subquotients are trivial. In particular, the restrictions to $G$ and $\Gamma$ are locally unipotent.
\end{theorem}

\begin{proof}
We give the proof for $\C[w]$ and $W_m$. The argument for $\Opoly$ and $W_m^{(p)}$ is identical.

\textbf{Rationality.} By Construction~\ref{con:Wm}, $\C[w]$ is the directed union of the finite-dimensional $\GC$-submodules $W_m$, and $\GC$ acts on each of them through the morphism of algebraic groups $\alpha_m$. Let $V'\subseteq\C[w]$ be any finite-dimensional $\GC$-stable subspace. A basis of $V'$ lies in some $W_m$, so $V'\subseteq W_m$, and $\GC$ acts on $V'$ through the restriction of $\alpha_m$, which is again a morphism of algebraic groups. Hence $\C[w]$ is a locally finite rational $\GC$-module in the sense of Section~\ref{sec:rational}.

\textbf{$\GC$ is unipotent.} The group $\GC$ is connected and nilpotent, so $\GC=T\times(\GC)_u$ as algebraic groups, with $T$ a torus and $(\GC)_u$ the unipotent part of $\GC$ \cite[Chapter~III, Theorem~10.6]{Bor91}. In exponential coordinates, which identify $\GC$ with $\g_\C$ as an affine variety because $G$ is nilpotent, $\C[\GC]=\C[\g_\C]$ is a polynomial ring, whose only units are the constants. Since $\GC\cong T\times(\GC)_u$ as varieties, the projection $\GC\to T$ pulls back the coordinate functions of a nontrivial torus $T\cong(\C^*)^r$ to nonconstant units of $\C[\GC]$. Hence $T=\{e\}$ and $\GC=(\GC)_u$.

\textbf{Unipotence of the representations.} For a morphism $\alpha$ of affine algebraic groups one has $\alpha(H_u)=\alpha(H)_u$ \cite[Chapter~I, Section~4.5]{Bor91}. Applied to $\alpha_m$, this gives $\alpha_m(\GC)=\alpha_m(\GC)_u$, that is, the image of $\GC$ is a unipotent subgroup of $\GL(W_m)$. By \cite[Chapter~I, Section~4.8]{Bor91}, it is conjugate to a subgroup of the group of upper triangular unipotent matrices, hence stabilizes a complete flag of $W_m$ with trivial one-dimensional subquotients. The same holds for the restriction to any finite-dimensional stable subspace $V'\subseteq W_m$, so $\C[w]$ is locally unipotent.

Restricting along $\Gamma\subset G\subset\GC$, the operators $\sigma(\gamma)=\rho(\gamma^{-1})^*$, and likewise $\sigma(g)$ for $g\in G$, act unipotently on each $W_m$.
\end{proof}

\subsection{Proof of Theorem~\ref{thm:E2}.}

Our result is a consequence of a chain of isomorphisms in cohomology. The first link is the following comparison, which is where the unipotence provided by Theorem~\ref{thm:unip} is used.

\begin{lemma}\label{lemma:devissage}
Let $V$ be a finite-dimensional unipotent rational $\GC$-module. Then there is an isomorphism
\[
j_V\colon H^\bullet(\g,V)\xrightarrow{\ \sim\ }H^\bullet(\Gamma,V),
\]
natural in $V$.
\end{lemma}

\begin{proof}
The restriction to $G$ of the representation on $V$ is unipotent,
so van Est's comparison theorem
\cite[Corollary~7.27]{Rag72} gives the asserted isomorphism.
Naturality in $V$ follows from the construction of the comparison
map in \cite[Proposition~7.18 and the proof of Theorem~7.26]{Rag72}.
\end{proof}

The second comparison is due to Hochschild. We quote it in the form we use.

\begin{fact}[{\cite[Theorem~5.1]{Hoc61}}]\label{fact:hochschild}
Let $U$ be a unipotent linear algebraic group over a field of characteristic zero, with Lie algebra $\mathfrak u$, and let $W$ be a rational $U$-module, regarded as a $\mathfrak u$-module by differentiation. Then there is a canonical isomorphism
\[
\Hrat^\bullet(U,W)\cong H^\bullet(\mathfrak u,W).
\]
\end{fact}

Let $\m$ be as in Section~\ref{sec:invdolbeault}, regarded as a rational
$\Gzo$-module via Proposition~\ref{prop:model}~(b). 

\begin{theorem}\label{thm:chain}
For all $(p,q)$ there is a chain of isomorphisms
\begin{align*}
H^q(\Gamma,\Opoly)
\ &\cong \ H^q(\g,\Opoly)
\ \cong \ \Hrat^q\big(\GC,\,\ind_{\Gzo}^{\GC} \m \big)\\
&\cong \ \Hrat^q\big(\Gzo,\, \m \big)
\ \cong \ H^q\big(\g^{0,1}, \m \big)\ =\ H^{p,q}_{\dbar}(\g,J).
\end{align*}
In particular, Theorem~\ref{thm:E2} holds.
\end{theorem}

\begin{proof}
\textbf{(a)} $H^q(\Gamma,\Opoly)\cong H^q(\g,\Opoly)$. By Construction~\ref{con:Wm} and Theorem~\ref{thm:unip}, $\Opoly=\varinjlim_m W_m^{(p)}$ is the directed union of finite-dimensional unipotent rational $\GC$-modules. Both cohomologies commute with this directed colimit. For $H^\bullet(\g,-)$, the Chevalley--Eilenberg complex $C^\bullet(\g,-)=\mathrm{Hom}_\R(\Lambda^\bullet_\R\g,-)$ commutes with directed colimits because each $\Lambda^q_\R\g$ is finite-dimensional. Its cohomology also commutes with these colimits, since directed colimits are exact. For $H^\bullet(\Gamma,-)$, the compact nilmanifold $M=\Gamma\backslash G$ is a $K(\Gamma,1)$ admitting a finite CW structure, so $\C$ admits a resolution $P_\bullet\to\C$ by finitely generated free $\C\Gamma$-modules \cite[Chapter~VIII, Proposition~6.3]{Bro82}. Then $H^q(\Gamma,-)=H^q(\mathrm{Hom}_{\C\Gamma}(P_\bullet,-))$, and each $\mathrm{Hom}_{\C\Gamma}(P_q,-)$ commutes with directed colimits since $P_q$ is finitely generated. Thus
\[
H^q(\g,\Opoly)=\varinjlim_m H^q\big(\g,W_m^{(p)}\big)\quad\text{and}\quad H^q(\Gamma,\Opoly)=\varinjlim_m H^q\big(\Gamma,W_m^{(p)}\big).
\]
By Lemma~\ref{lemma:devissage}, each $j_{W_m^{(p)}}$ is an isomorphism, and these isomorphisms are compatible with the inclusions $W_m^{(p)}\subseteq W_{m'}^{(p)}$ by naturality. Their colimit is an isomorphism $H^q(\g,\Opoly)\cong H^q(\Gamma,\Opoly)$.
\medskip

\textbf{(b)} $H^q(\g,\Opoly)\cong\Hrat^q(\GC,\ind_{\Gzo}^{\GC}\m)$. Real and complex Lie-algebra cohomology agree on a complex module. For a $\C$-vector space $V$, restriction along $\g\hookrightarrow\g_\C$ gives
\[
C^q(\g_\C,V)=\mathrm{Hom}_\C(\Lambda^q_\C\g_\C,V)\ \xrightarrow{\ \sim\ }\ \mathrm{Hom}_\R(\Lambda^q_\R\g,V)=C^q(\g,V),
\]
because $\Lambda^q_\C\g_\C=(\Lambda^q_\R\g)\otimes_\R\C$ and an $\R$-linear map into the complex space $V$ extends uniquely $\C$-linearly. The $\g_\C$-action on $\Opoly$, obtained by differentiating the algebraic $\GC$-action of Proposition~\ref{prop:model}~(a), is $\C$-linear and restricts on $\g$ to the differential of the $G$-action used in (a). Hence the two differentials correspond, and $H^q(\g,\Opoly)=H^q(\g_\C,\Opoly)$.

By Theorem~\ref{thm:unip}, $\Opoly$ is a rational $\GC$-module, and $\GC$ is unipotent. Fact~\ref{fact:hochschild} with $U=\GC$ and $W=\Opoly$ gives
\[
\Hrat^q(\GC,\Opoly)\cong H^q(\g_\C,\Opoly).
\]
Finally, $\Opoly\cong\ind_{\Gzo}^{\GC}\m$ as rational $\GC$-modules by Proposition~\ref{prop:model}~(b).
\medskip

\textbf{(c)} $\Hrat^q(\GC,\ind_{\Gzo}^{\GC}\m)\cong\Hrat^q(\Gzo,\m)$. Algebraic induction $\ind_{\Gzo}^{\GC}$ is right adjoint to the exact restriction functor \cite[Proposition~1.4(a)]{CPS77}, hence it is left exact and sends rationally injective modules to rationally injective modules \cite[Section~1]{CPS77}. Frobenius reciprocity gives $(\ind_{\Gzo}^{\GC}W)^{\GC}=W^{\Gzo}$, that is,
\[
(-)^{\GC}\circ\ind_{\Gzo}^{\GC}=(-)^{\Gzo}
\]
as functors on rational $\Gzo$-modules.
Exactness of induction also follows directly from the algebraic
product decomposition of Proposition~\ref{prop:model}~(a).
Indeed, restriction to $G^{1,0}$ gives, naturally in every rational
$\Gzo$-module $W$, an isomorphism of vector spaces
\[
\ind_{\Gzo}^{\GC}W
\longrightarrow
\C[G^{1,0}]\otimes_{\C}W.
\]
Its inverse extends a regular map $f\colon G^{1,0}\to W$ by
\[
\widetilde f(ab)=b^{-1}\cdot f(a),
\qquad a\in G^{1,0},\quad b\in\Gzo.
\]
The algebraic product decomposition and rationality of $W$ ensure
that $\widetilde f$ is regular. Since tensoring over $\C$ is exact,
the induction functor is exact.
Hence $R^t\ind_{\Gzo}^{\GC}=0$ for $t>0$. Since the categories of rational modules have enough injectives (Section~\ref{sec:rational}), the Grothendieck spectral sequence \cite{Gro57} of the composite above reads
\[
E_2^{s,t}=\Hrat^s\big(\GC,\,R^t\ind_{\Gzo}^{\GC}\m\big)\ \Longrightarrow\ \Hrat^{s+t}(\Gzo,\m),
\]
and collapses to the edge isomorphism $\Hrat^q(\GC,\ind_{\Gzo}^{\GC}\m)\cong\Hrat^q(\Gzo,\m)$, see also \cite[Remark~2.5(c)]{CPS77}.
\medskip

\textbf{(d)} $\Hrat^q(\Gzo,\m)\cong H^q(\g^{0,1},\m)=H^{p,q}_{\dbar}(\g,J)$. The group $\Gzo=\exp\g^{0,1}$ is a unipotent algebraic subgroup of $\GC$ with Lie algebra $\g^{0,1}$, and $\m$ is a finite-dimensional rational $\Gzo$-module. Fact~\ref{fact:hochschild} gives $\Hrat^q(\Gzo,\m)\cong H^q(\g^{0,1},\m)$.

We identify the module and the differential. By Proposition~\ref{prop:model}~(b), $\m=\Lambda^p(\g_\C/\g^{0,1})^*$, with $\Gzo$ acting through the isotropy representation of Remark~\ref{rem:isotropy}. For $h\in\Gzo$ one has $\Ad(h)\g^{0,1}=\g^{0,1}$, so $\Ad(h)$ induces a linear automorphism $\Ad_{\g_\C/\g^{0,1}}(h)$ of $\g_\C/\g^{0,1}$, and $h\cdot\lambda=\lambda\circ\Ad_{\g_\C/\g^{0,1}}(h)^{-1}$. Since $\frac{d}{dt}\big|_{t=0}\Ad(\exp t\bar X)Y=[\bar X,Y]$, the differential of the isotropy action on
$\g_\C/\g^{0,1}$ is
\[
\bar X\cdot\big(Y+\g^{0,1}\big)=[\bar X,Y]+\g^{0,1},\qquad \bar X\in\g^{0,1},\ Y\in\g_\C,
\]
which is well defined because $\g^{0,1}$ is a subalgebra. Under the linear isomorphism $\g_\C/\g^{0,1}\cong\g^{1,0}$, $Y+\g^{0,1}\mapsto\pr^{1,0}Y$, it reads $\bar X\cdot Y=\pr^{1,0}[\bar X,Y]$. The induced action on $\m$ is therefore
\[
(\bar X\cdot\lambda)(Y_1,\dots,Y_p)
=
-\sum_{a=1}^p
\lambda\big(
Y_1,\dots,\pr^{1,0}[\bar X,Y_a],\dots,Y_p
\big),
\]
for $\lambda\in\m$ and $Y_1,\dots,Y_p\in\g^{1,0}$.
This is precisely the module structure introduced in
Section~\ref{sec:invdolbeault}, so
\eqref{eq:CEidentification} gives
\[
C^q(\g^{0,1},\m)=\Lambda^q(\g^{0,1})^*\otimes\Lambda^p(\g^{1,0})^*=\Lambda^{p,q}\g_\C^* .
\]
Under this identification, the Chevalley--Eilenberg differential
\[
(d_{CE}\omega)(\bar X_0,\dots,\bar X_q)=\sum_i(-1)^i\,\bar X_i\cdot\omega(\dots,\widehat{\bar X_i},\dots)+\sum_{i<j}(-1)^{i+j}\omega\big([\bar X_i,\bar X_j],\dots\big)
\]
is the operator $\dbar$. On left-invariant forms,
$d$ is the Chevalley--Eilenberg differential of $\g_\C$,
and $\dbar$ is its component of bidegree $(0,1)$.
To see this, evaluate $d\omega$ on
$(\bar X_0,\dots,\bar X_q,Y_1,\dots,Y_p)$,
with $\bar X_i\in\g^{0,1}$ and $Y_a\in\g^{1,0}$,
using \eqref{eq:CEidentification}.
For a mixed bracket, only its $(1,0)$-component contributes.
Moving it to the position of $Y_a$ gives the total sign
\[
(-1)^{i+q+a}(-1)^{q+a-1}=(-1)^{i+1}.
\]
By the module-action formula above, the mixed-bracket contribution is therefore
\[
\sum_{i=0}^{q}(-1)^i
\bigl(\bar X_i\cdot
\omega(\bar X_0,\dots,\widehat{\bar X_i},\dots,\bar X_q)\bigr)
(Y_1,\dots,Y_p).
\]
The brackets $[\bar X_i,\bar X_j]$ give the second sum in $d_{CE}$.
All remaining terms vanish by bidegree and integrability.
Thus $d_{CE}$ corresponds to $\dbar$.
Hence
$H^q(\g^{0,1},\m)=H^{p,q}_{\dbar}(\g,J)$,
of dimension $h^{p,q}(\g,J)$.

Combining (a)--(d) proves the asserted chain of isomorphisms
and hence Theorem~\ref{thm:E2}.
\end{proof}

\section{Proof of Theorem~\ref{thm:main}}\label{sec:proof}

We are ready to combine the surjectivity of Section~\ref{sec:analytic}, the dimension count of Section~\ref{sec:algebraic}, and the injectivity of $\iota$ due to Console and Fino.

\begin{proof}[Proof of Theorem~\ref{thm:main}]
Compactness of $M$ implies that each $H^{p,q}_{\dbar}(M)$ is finite-dimensional. By Theorem~\ref{thm:E1}, the map from $H^q(\Gamma,\Opoly)$ to $H^{p,q}_{\dbar}(M)$ is surjective, and, by Theorem \ref{thm:E2}, we have the equality $\dim H^q(\Gamma,\Opoly) = h^{p,q} (\g, J)$. Hence, we get
\[
h^{p,q}(M)\ \le\ \dim H^q(\Gamma,\Opoly) = h^{p,q} (\g, J).
\]
The opposite inequality follows from injectivity of the map $\iota \colon H^{p,q}_{\dbar}(\g,J)\to H^{p,q}_{\dbar}(M)$ proved in \cite[Lemma~9]{CF01}. Therefore, $h^{p,q}(M)=h^{p,q}(\g,J)$, and the injection $\iota$ between finite-dimensional spaces of the same dimension is an isomorphism.
\end{proof}

\section{Consequences and applications}\label{sec:consequences}

In this section, we collect results that follow from Theorem~\ref{thm:main}.

\subsection{Deformation theory.}\label{sec:cons:tangent}
We recall the Kuranishi construction in the form used in
\cite[proof of Theorem~2.6]{RolJLMS}. The Kodaira--Spencer DGLA is
\[
L^\bullet=A^{0,\bullet}(T^{1,0}M).
\]
Fix a Hermitian metric, and let $\mathsf H$ and $\mathsf G$ denote
the harmonic projection and Green operator for its $\dbar$-Laplacian.
With the bracket convention in which integrability reads
$\dbar\phi+\tfrac12[\phi,\phi]=0$, a basis
$\theta_1,\dots,\theta_m$ of harmonic elements in $L^1$ determines a convergent solution
near $0\in\C^m$ of
\begin{equation}\label{eq:kuranishi:solution}
\phi(t)=\sum_{j=1}^m t_j\theta_j
-\frac12\dbar^*\mathsf G[\phi(t),\phi(t)].
\end{equation}
The obstruction equations
\begin{equation}\label{eq:kuranishi:obstructions}
\mathsf H[\phi(t),\phi(t)]=0
\end{equation}
define the Kuranishi base $(B,0)$ as a complex analytic germ,
which may be singular. The corresponding family is complete.
We will show that this construction takes place in the
finite-dimensional invariant sub-DGLA.

\begin{lemma}\label{thm:tangent}
Let $M=\Gamma\backslash G$ be a compact nilmanifold with a
left-invariant complex structure $J$ and $\dim_\C M=n$.
For all $(p,q)$, the inclusion of invariant sections induces an isomorphism
\[
H^q\big(\g^{0,1},\Lambda^p\g^{1,0}\big)
\;\xrightarrow{\ \sim\ }\;
H^q\big(M,\Lambda^pT^{1,0}M\big),
\]
where the coefficient action is induced by
$\bar X\cdot Y=\pr^{1,0}[\bar X,Y]$.
In particular, the inclusion of Kodaira--Spencer DGLAs
\[
\big(\Lambda^\bullet(\g^{0,1})^*\otimes\g^{1,0},
\dbar,[\,\cdot\,,\cdot\,]\big)
\hookrightarrow
\big(A^{0,\bullet}(T^{1,0}M),\dbar,[\,\cdot\,,\cdot\,]\big)
\]
is a quasi-isomorphism.
\end{lemma}

\begin{proof}
By \cite[Theorem~2.7]{BDV09}, the canonical bundle of $M$ admits
a nowhere-vanishing invariant holomorphic section $\eta$.
Contraction with $\eta$ gives an isomorphism of holomorphic bundles
\[
c\colon\Lambda^pT^{1,0}M\longrightarrow\Omega^{n-p}_M,
\qquad v\longmapsto\iota_v\eta.
\]
Its extensions
\[
\begin{aligned}
c_q\colon A^{0,q}(\Lambda^pT^{1,0}M)&\longrightarrow A^{n-p,q}(M),\\
\beta\otimes v&\longmapsto\beta\wedge\iota_v\eta,
\end{aligned}
\]
give an isomorphism of Dolbeault complexes that identifies the
invariant subcomplexes. The cohomology assertion follows from
Theorem~\ref{thm:main} in bidegree $(n-p,q)$.
Since the Kodaira--Spencer bracket preserves invariant sections,
the inclusion is a morphism of DGLAs, and the case $p=1$ makes
it a quasi-isomorphism.
\end{proof}

We can now prove Hasegawa's conjecture on small deformations
\cite{Has10}, together with a uniform polynomial realization of the universal
coverings of the nearby fibers.

\begin{theorem}\label{thm:hasegawa}
Let $M=\Gamma\backslash G$ be a compact nilmanifold of complex
dimension $n$ with a left-invariant complex structure $J$, and let
$s$ be the nilpotency step of $G$.
The Kuranishi germ and family of $M$ are obtained by the construction
\eqref{eq:kuranishi:solution}--\eqref{eq:kuranishi:obstructions}
in the invariant Kodaira--Spencer DGLA.
After shrinking a representative $\pi\colon\mathcal X\to B$ of this
family, its fibers satisfy
\[
M_t\coloneqq\pi^{-1}(t)\cong(\Gamma\backslash G,J_t)
\cong\Gamma_t\backslash\C^n,
\qquad J_0=J,
\]
where $J_t$ is left-invariant and
$\Gamma_t\subset\Aut_{\mathrm{pol}}(\C^n)$ satisfies
\[
\Gamma_t\cong\Gamma,
\qquad \Gamma_0=\rho(\Gamma),
\qquad \deg\gamma\le s^{2n}
\quad(t\in B,\ \gamma\in\Gamma_t).
\]
Each $\Gamma_t$ is a uniform lattice in a real Lie group $G_t\subset\Aut_{\mathrm{pol}}(\C^n)$ acting simply transitively on $\C^n$. Every sufficiently small deformation of $M$ is
biholomorphic to one of these fibers.
\end{theorem}

\begin{proof}
Choose an invariant Hermitian metric and write
\[
V^\bullet
\coloneqq
\Lambda^\bullet(\g^{0,1})^*\otimes\g^{1,0}
\subset A^{0,\bullet}(T^{1,0}M).
\]
Both $\dbar$ and its formal adjoint preserve $V^\bullet$.
For the adjoint, this follows by integration by parts in invariant
frames, since $G$ is unimodular and invariant vector fields have
zero divergence. Thus the Dolbeault Laplacian preserves
$V^\bullet$. Its restriction is the Laplacian of the
finite-dimensional Hermitian complex $V^\bullet$, so
finite-dimensional Hodge theory identifies
\[
V^q\cap\ker\Delta_{\dbar}\simeq H^q(V^\bullet).
\]
By Lemma~\ref{thm:tangent}, this space has the same dimension as
the full harmonic space in degree $q$. Hence every harmonic
tensor is invariant. Moreover, self-adjointness of the Laplacian
shows that the harmonic projection $\mathsf H$ and the Green
operator $\mathsf G$ preserve $V^\bullet$.

Since the Kodaira--Spencer bracket preserves invariant tensors,
the recursive solution of \eqref{eq:kuranishi:solution} and the
obstruction equations \eqref{eq:kuranishi:obstructions} lie in
the invariant sub-DGLA. As in \cite[Theorem~2.6]{RolJLMS}, this
gives the Kuranishi family in terms of left-invariant complex
structures $J_t$ on $G$. Completeness accounts for every
sufficiently small deformation.

For each $t\in B$, choose $\Phi_t\colon(G,J_t)\to\C^n$ by the
construction of Section~\ref{sec:setup}, with $\Phi_0=\Phi$, and set
\[
\rho_t(g)=\Phi_t\circ L_g\circ\Phi_t^{-1},
\qquad G_t=\rho_t(G),\qquad\Gamma_t=\rho_t(\Gamma).
\]
Proposition~\ref{prop:uniform:degree} gives
$\deg\rho_t(g)\le s^{2n}$ for all $t$ and $g\in G$.
Its inverse is the polynomial map $\rho_t(g^{-1})$.
Transport the real Lie group structure of $G$ to $G_t$ via
$\rho_t$. Then $G_t$ acts simply transitively, $\Gamma_t$ is
discrete and cocompact in $G_t$, and $\Gamma_t\cong\Gamma$. Since $\Phi_t$ conjugates the action of $\Gamma$ on $G$ by left translations to the action of $\Gamma_t$ on $\C^n$, the latter is free, properly discontinuous and cocompact.
Finally, $\Phi_t$ descends to the asserted biholomorphism.
\end{proof}

\subsection{The double complex of invariant forms.}\label{sec:cons:double}
Theorem~\ref{thm:main}, together with complex conjugation,
implies that the inclusion of invariant forms is an
$E_1$-isomorphism of double complexes. Consequently, every
invariant preserved by $E_1$-isomorphisms is computed by
invariant forms.

Let $(A^{\bullet,\bullet},\partial,\dbar)$ be a double complex, with the two spectral sequences whose first pages are $E_1^{p,q}=H^{p,q}_{\dbar}(A)$ and ${}'E_1^{p,q}=H^{p,q}_{\partial}(A)$.

\begin{definition}[\cite{Ste21}]\label{def:E1iso}
A morphism $f\colon A\to B$ of double complexes is an \emph{$E_1$-isomorphism} if the induced maps
\[
H^{p,q}_{\partial}(f)\colon H^{p,q}_{\partial}(A)\longrightarrow H^{p,q}_{\partial}(B) \quad \text{and} \quad
H^{p,q}_{\dbar}(f)\colon H^{p,q}_{\dbar}(A)\longrightarrow H^{p,q}_{\dbar}(B),
\]
are isomorphisms for all $(p,q)$ or, equivalently, if $f$ induces an isomorphism on the first page of each of the two spectral sequences above.
\end{definition}

For bounded double complexes, both filtrations are bounded, so an $E_1$-isomorphism induces isomorphisms (compatible with all differentials $d_r$) on every page $E_r$, $1\le r\le\infty$ of both spectral sequences, as well as on the abutment.

Let $A^{\bullet,\bullet}_{\g} \coloneqq \Lambda^{\bullet,\bullet}\g_\C^*$ be the Chevalley--Eilenberg double complex of $(\g,J)$ and $A^{\bullet,\bullet}(M)$ the double complex of smooth forms, both endowed with the standard differentials $(\partial, \dbar)$.

\begin{theorem}\label{thm:E1iso}
For every compact nilmanifold with left-invariant complex structure, the inclusion
\[
\iota\colon\big(A^{\bullet,\bullet}_{\g},\partial,\dbar\big)\longrightarrow
\big(A^{\bullet,\bullet}(M),\partial,\dbar\big)
\]
is an $E_1$-isomorphism of bounded double complexes.
\end{theorem}

\begin{proof}
Complex conjugation is an antilinear involution of $A^{\bullet,\bullet}(M)$ which
exchanges $\partial$ and $\dbar$ and the bidegrees $(p,q)\leftrightarrow(q,p)$. Since
$\g$ is a real form of $\g_\C$, conjugation preserves $A^{\bullet,\bullet}_\g$ and commutes
with $\iota$. Hence $H^{p,q}_{\partial}\cong\overline{H^{q,p}_{\dbar}}$ on both sides,
compatibly with $\iota$, and the $\partial$-isomorphism follows from the
$\dbar$-isomorphism of Theorem~\ref{thm:main}.
\end{proof}

In the next statement, $H^{\bullet,\bullet}_{BC}$ and $H^{\bullet,\bullet}_{A}$ denote Bott--Chern and Aeppli cohomologies. The groups $\mathsf A^{p,q},\dots,\mathsf F^{p,q}$ are the six bigraded spaces of Varouchas \cite[Section 3.1]{Var86}, which measure the failure of the natural maps $H^{\bullet,\bullet}_{BC}\to H^{\bullet,\bullet}_{\dbar}\to H^{\bullet,\bullet}_{A}$ to be isomorphisms. Finally, $H^{k}_{S_{p,q}}$ is the cohomology of the complex $\mathcal L^{\bullet}_{p,q}$ of Bigolin--Schweitzer \cite{Big69, Big70, Pio25, Sch07}. All of these are functors of the double complex alone and vanish on squares, that is, they are cohomological functors in the sense of \cite{Ste22, Ste25}.

\begin{corollary}\label{cor:zigzag}
The double complexes $(A^{\bullet,\bullet}_\g, \partial, \dbar)$ and $(A^{\bullet,\bullet}(M), \partial, \dbar)$ have the same multiplicities of all zigzags in the decomposition of \cite{Ste21}. Consequently, every invariant of a bounded double complex that is determined by the zigzag multiplicities is computed by invariant forms. In particular, the inclusion $\iota$ induces isomorphisms
\[
H^{p,q}_{\star}(\g,J)\;\xrightarrow{\ \sim\ }\;H^{p,q}_{\star}(M),
\qquad \star\in\{\partial,\dbar,BC,A\},
\]
in every bidegree, and isomorphisms $E_r^{p,q}(\g,J)\cong E_r^{p,q}(M)$, $1\le r\le\infty$, of both spectral sequences, compatible with all differentials $d_r$ and with the abutment. The same holds for the six Varouchas groups and for the Bigolin--Schweitzer cohomologies.

In particular, for a fixed left-invariant complex structure
$J$ on $G$, any two lattices in $G$ give the same stated
invariants.
\end{corollary}

\begin{proof}
By Theorem~\ref{thm:E1iso}, the inclusion $\iota$ is an
$E_1$-isomorphism. Hence the two double complexes have
the same zigzag multiplicities by
\cite[Proposition~11]{Ste21}.
These multiplicities are finite: every zigzag contributes
to row or column cohomology, and both are finite-dimensional
by compactness of $M$.

The listed cohomology functors are linear and vanish on
squares. Thus \cite[Proposition~12]{Ste21} shows that
$\iota$ induces the asserted cohomology isomorphisms.
The comparisons of both spectral sequences and their
abutments follow from the spectral-sequence comparison
theorem, since the double complexes are bounded.

Finally, the invariant double complex depends only on
$(\g,J)$, so the stated invariants are independent of
the lattice.
\end{proof}

\begin{remark}
Theorem~\ref{thm:main} and the Bott--Chern and Aeppli cases
of Corollary~\ref{cor:zigzag}, together with Nomizu's theorem,
settle Angella's conjecture.
The Bott--Chern conclusion stated in the conjecture also follows from
\cite[Theorem~3.7]{Ang13}, whose de Rham and Dolbeault
comparison hypotheses hold here by Nomizu's theorem
and Theorem~\ref{thm:main}, respectively.
\end{remark}

\printbibliography

\end{document}